\documentclass[preprint,12pt,a4paper]{elsarticle}
\pdfoutput=1
\usepackage{lipsum}
\makeatletter
\def\ps@pprintTitle{%
 \let\@oddhead\@empty
 \let\@evenhead\@empty
 \def\@oddfoot{}%
 \let\@evenfoot\@oddfoot}
\makeatother

\usepackage{amssymb}
\usepackage{mathrsfs}
\usepackage{amsthm}
\usepackage[utf8]{inputenc}
\usepackage{enumerate}
\usepackage{amssymb} \usepackage{amsfonts} \usepackage{amsmath}
\usepackage{accents}
\usepackage{color}

\newtheorem{theorem}{Theorem}[section]
\newtheorem{corollary}[theorem]{Corollary}
\newtheorem{lemma}[theorem]{Lemma}
\newtheorem{property}[theorem]{Property}
\newtheorem{fact}[theorem]{Fact}

\theoremstyle{definition}
\newtheorem{defn}[theorem]{Definition}
\newtheorem{ex}[theorem]{Example}
\newtheorem{exs}[theorem]{Examples}

\theoremstyle{remark}
\newtheorem{remark}[theorem]{Remark}

\newcommand{\Leb} {\ensuremath{\textnormal{Leb}}}

\newcommand{\p} {\ensuremath {\mathbb{P}}}
\newcommand{\E} {\ensuremath {\mathbb{E}}}

\newcommand{\N} {\ensuremath {\mathbb{N}}}
\newcommand{\R} {\ensuremath {\mathbb{R}}}

\newcommand{\Qi} {\ensuremath {\mathscr{Q}}}

\newcommand{\I} {\ensuremath {\mathbb{I}}}
\newcommand{\F} {\ensuremath {\mathscr{F}}}

\newcommand{\Nn} {\ensuremath {\mathscr{Nn}}}

\newcommand{\A} {\ensuremath {\mathscr{A}}}
\newcommand{\X} {\ensuremath {\mathscr{X}}}

\newcommand{\B} {\ensuremath {\mathscr{B}}}

\newcommand{\D} {\ensuremath {\mathscr{D}}}

\newcommand{\Ci} {\ensuremath {\mathscr{C}}}
\newcommand{\mo} {\ensuremath {\mathscr{P}}}

\newcommand{\Wh} {\ensuremath {\mathscr{W}}}

\newcommand{\tV} {\ensuremath {\accentset{\triangle}{V}}}

\begin{document}

\begin{frontmatter}



\title{Asymptotic equivalence for pure jump L\'evy processes with unknown L\'evy density and Gaussian white noise}


\author[LJK,corr]{Ester Mariucci}

\address[LJK]{\it Laboratoire LJK, Universit\'e Joseph Fourier UMR 5224 51, Rue des Math\'ematiques, Saint Martin d'H\`eres BP 53 38041 Grenoble Cedex 09}
\address[corr]{Corresponding Author, Ester.Mariucci@imag.fr}

\begin{abstract}
The aim of this paper is to establish a global asymptotic equivalence between the experiments generated by the discrete (high frequency) or continuous observation of a path of a L\'evy process and a Gaussian white noise experiment observed up to a time $T$, with $T$ tending to $\infty$. These approximations are given in the sense of the Le Cam distance, under some smoothness conditions on the unknown L\'evy density. All the asymptotic equivalences are established by constructing explicit Markov kernels that can be used to reproduce one experiment from the other.

\end{abstract}


\end{frontmatter}


\section{Introduction}
Lévy processes are a fundamental tool in modelling situations, like the dynamics of asset prices and weather measurements, where sudden changes in values may happen. For that reason they are widely employed, among many other fields, in mathematical finance. To name a simple example, the price of a commodity at time $t$ is commonly given as an exponential function of a Lévy process. In general, exponential Lévy models are proposed for their ability to take into account several empirical features observed in the returns of assets such as heavy tails, high-kurtosis and asymmetry (see \cite{tankov} for an introduction to financial applications). 

From a mathematical point of view, Lévy processes are a natural extension of the Brownian motion which preserves the tractable statistical properties of its increments, while relaxing the continuity of paths. The jump dynamics of a Lévy process is dictated by its Lévy density, say $f$. If $f$ is continuous, its value at a point $x_0$ determines how frequent jumps of size close to $x_0$ are to occur per unit time. Concretely, if $X$ is a pure jump Lévy process with Lévy density $f$, then the function $f$ is such that 
 $$\int_Af(x)dx=\frac{1}{t}\E\bigg[\sum_{s\leq t}\I_A(\Delta X_s)\bigg],$$
for any Borel set $A$ and $t>0$. Here, $\Delta X_s\equiv X_s-X_{s^-}$ denotes the magnitude of the jump of $X$ at time $s$ and $\I_A$ is the characteristic function. Thus, the Lévy measure
$$\nu(A):=\int_A f(x)dx,$$
is the average number of jumps (per unit time) whose magnitudes fall in the set $A$.
Understanding the jumps behavior, therefore requires to estimate the Lévy measure. Several recent works have treated this problem, see e.g. \cite{bel15} for an overview.

 When the available data consists of the whole trajectory of the process during a time interval $[0,T]$, the problem of estimating $f$ may be reduced to estimating the intensity function of an inhomogeneous Poisson process (see, e.g. \cite{fig06, rey03}). However, a continuous-time sampling is never available in practice and thus the relevant problem is that of estimating $f$ based on discrete sample data $X_{t_0},\dots,X_{t_n}$ during a time interval $[0,T_n]$. In that case, the jumps are latent (unobservable) variables and
 that clearly adds to the difficulty of the problem. 
From now on we will place ourselves in a high-frequency setting, that is we assume that the sampling interval $\Delta_n=t_i-t_{i-1}$ tends to zero as $n$ goes to infinity. Such a high-frequency based statistical approach has played a central role in the recent literature on nonparametric estimation for Lévy processes (see e.g. \cite{fig09,comte10,comte11,bec12,duval12}).
Moreover, in order to make consistent estimation possible, we will also ask the observation time $T_n$ to tend to infinity in order to allow the identification of the jump part in the limit.

Our aim is to prove that, under suitable hypotheses, estimating the Lévy density $f$ is equivalent to estimating the drift of an adequate Gaussian white noise model. In general, asymptotic equivalence results for statistical experiments provide a deeper understanding of statistical problems and allow to single out their main features. The idea is to pass via asymptotic equivalence to another experiment which is easier to analyze. 
 By definition, two sequences of experiments $\mo_{1,n}$ and $\mo_{2,n}$, defined on possibly different sample spaces, but with the same parameter set, are asymptotically equivalent if the Le Cam distance $\Delta(\mo_{1,n},\mo_{2,n})$ tends to zero. For $\mo_{i}=(\X_i,\A_i, \big(P_{i,\theta}:\theta\in\Theta)\big)$, $i=1,2$, $\Delta(\mo_1,\mo_2)$ is the symmetrization of the deficiency $\delta(\mo_1,\mo_2)$ where 
 $$\delta(\mo_{1},\mo_{2})=\inf_K\sup_{\theta\in\Theta}\big\|KP_{1,\theta}-P_{2,\theta}\big\|_{TV}.$$ Here the infimum is taken over all randomizations from $(\X_1,\A_1)$ to $(\X_2,\A_2)$ and $\| \cdot \|_{TV}$ denotes the total variation distance.
Roughly speaking, the Le Cam distance quantifies how much one fails to reconstruct (with the help of a randomization) a model from the other one and vice versa. Therefore, we say that $\Delta(\mo_1,\mo_2)=0$ can be interpreted as ``the models $\mo_1$ and $\mo_2$ contain the same amount of information about the parameter $\theta$.'' The general definition of randomization is quite involved but, in the most frequent examples (namely when the sample spaces are Polish and the experiments dominated), it reduces to that of a Markov kernel.
One of the most important feature of the Le Cam distance is that it can be also interpreted in terms of statistical decision theory (see \cite{lecam, LC2000}; a short review is presented in the Appendix). As a consequence, saying that two statistical models are equivalent means that any statistical inference procedure can be transferred from one model to the other in such a way that the asymptotic risk remains the same, at least for bounded loss functions. Also, as soon as two models, $\mo_{1,n}$ and $\mo_{2,n}$, that share the same parameter space $\Theta$ are proved to be asymptotically equivalent, the same result automatically holds for the restrictions of both $\mo_{1,n}$ and $\mo_{2,n}$ to a smaller subclass of $\Theta$.

Historically, the first results of asymptotic equivalence in a nonparametric context date from 1996 and are due to \cite{BL} and \cite{N96}.
The first two authors have shown the asymptotic equivalence of nonparametric regression and a Gaussian white noise model while the third one those of density estimation and white noise. Over the years many generalizations of these results have been proposed such as \cite{regression02,GN2002,ro04,C2007,cregression,R2008,C2009,R2013,schmidt14} for nonparametric regression or \cite{cmultinomial,j03,BC04} for nonparametric density estimation models.
Another very active field of study is that of diffusion experiments. The first result of equivalence between diffusion models and Euler scheme was established in 1998, see \cite{NM}. In later papers generalizations of this result have been considered (see \cite{C14, esterdiffusion}).
Among others we can also cite equivalence results for generalized linear models \cite{GN}, time series \cite{GN2006,NM}, diffusion models \cite{D,CLN,R2006,rmultidimensionale}, GARCH model \cite{B}, functional linear regression \cite{M2011}, spectral density estimation \cite{GN2010} and volatility estimation \cite{R11}. Negative results are somewhat harder to come by; the most notable among them are \cite{sam96,B98,wang02}.
There is however a lack of equivalence results concerning processes with jumps. A first result in this sense is \cite{esterESAIM} in which global asymptotic equivalences between the experiments generated by the discrete or continuous observation of a path of a Lévy process and a Gaussian white noise experiment are established. More precisely, in that paper, we have shown that estimating the drift function $h$ from a continuously or discretely (high frequency) time inhomogeneous jump-diffusion process:
\begin{equation}\label{ch4X}
 X_t=\int_0^th(s)ds+\int_0^t\sigma(s)dW_s +\sum_{i=1}^{N_t}Y_i,\quad t\in[0,T_n],
\end{equation}
is asymptotically equivalent to estimate $h$ in the Gaussian model:
\begin{equation*}
 dy_t=h(t)dt+\sigma(t)dW_t, \quad t\in[0,T_n].
\end{equation*}

Here we try to push the analysis further and we focus on the case in which the considered parameter is the Lévy density and $X=(X_t)$ is a pure jump Lévy process (see \cite{carr02} for the interest of such a class of processes when modelling asset returns).
More in details, we consider the problem of estimating the Lévy density (with respect to a fixed, possibly infinite, Lévy measure $\nu_0$ concentrated on $I\subseteq \R$) $f:=\frac{d\nu}{d\nu_0}:I\to \R$ from a continuously or discretely observed pure jump Lévy process $X$ with possibly infinite Lévy measure. Here $I\subseteq \R$ denotes a possibly infinite interval and $\nu_0$ is supposed to be absolutely continuous with respect to Lebesgue with a strictly positive density $g:=\frac{d\nu_0}{d\Leb}$. In the case where $\nu$ is of finite variation one may write: 
\begin{equation}\label{eqn:ch4Levy}
X_t=\sum_{0<s\leq t}\Delta X_s
\end{equation}
or, equivalently, $X$ has a characteristic function given by:
$$\E\big[e^{iuX_t}\big]=\exp\bigg(-t\bigg(\int_{I}(1-e^{iuy})\nu(dy)\bigg)\bigg).$$

We suppose that the function $f$ belongs to some a priori set $\F$, nonparametric in general. 
The discrete observations are of the form $X_{t_i}$, where $t_i=T_n\frac{i}{n}$, $i=0,\dots,n$ with $T_n=n\Delta_n\to \infty$ and $\Delta_n\to 0$ as $n$ goes to infinity. We will denote by $\mo_n^{\nu_0}$ the statistical model associated with the continuous observation of a trajectory of $X$ until time $T_n$ (which is supposed to go to infinity as $n$ goes to infinity) and by $\Qi_n^{\nu_0}$ the one associated with the observation of the discrete data $(X_{t_i})_{i=0}^n$.
The aim of this paper is to prove that, under adequate hypotheses on $\F$ (for example, $f$ must be bounded away from zero and infinity; see Section \ref{subsec:ch4parameter} for a complete definition), the models $\mo_n^{\nu_0}$ and $\Qi_n^{\nu_0}$ are both asymptotically equivalent to a sequence of Gaussian white noise models of the form: 
$$dy_t=\sqrt{f(t)}dt+\frac{1}{2\sqrt{T_n}}\frac{dW_t}{\sqrt{g(t)}},\quad t\in I.$$
As a corollary, we then get the asymptotic equivalence between $\mo_n^{\nu_0}$ and $\Qi_n^{\nu_0}$.
The main results are precisely stated as Theorems \ref{ch4teo1} and \ref{ch4teo2}.
A particular case of special interest arises when $X$ is a compound Poisson process, $\nu_0\equiv \Leb([0,1])$ and $\F\subseteq \F_{(\gamma,K,\kappa,M)}^I$ where, for fixed $\gamma\in (0,1]$ and $K,\kappa, M$ strictly positive constants,  $\F_{(\gamma,K,\kappa,M)}^I$ is a class of continuously differentiable functions on $I$ defined as follows:
\begin{equation}\label{ch4:fholder}
 \F_{(\gamma,K,\kappa,M)}^I=\Big\{f: \kappa\leq f(x)\leq M, \ |f'(x)-f'(y)|\leq K|x-y|^{\gamma},\ \forall x,y\in I\Big\}.
\end{equation}
In this case, the statistical models $\mo_n^{\nu_0}$ and $\Qi_n^{\nu_0}$ are both equivalent to the Gaussian white noise model:
$$dy_t=\sqrt{f(t)}dt+\frac{1}{2\sqrt{T_n}}dW_t,\quad t\in [0,1].$$
See Example \ref{ex:ch4CPP} for more details.
By a theorem of Brown and Low in \cite{BL}, we obtain, a posteriori, an asymptotic equivalence with the regression model 
$$Y_i=\sqrt{f\Big(\frac{i}{T_n}\Big)}+\frac{1}{2\sqrt{T_n}}\xi_i, \quad \xi_i\sim\Nn(0,1), \quad i=1,\dots, [T_n].$$
Note that a similar form of a Gaussian shift was found to be asymptotically equivalent to a nonparametric density estimation experiment, see \cite{N96}.
Let us mention that we also treat some explicit examples where $\nu_0$ is neither finite nor compactly-supported (see Examples \ref{ch4ex2} and \ref{ex3}).

Without entering into any detail, we remark here that the methods are very different from those in \cite{esterESAIM}. In particular, since $f$ belongs to the discontinuous part of a Lévy process, rather then its continuous part, the Girsanov-type changes of measure are irrelevant here. We thus need new instruments, like the Esscher changes of measure.

Our proof is based on the construction, for any given Lévy measure $\nu$, of two adequate approximations $\hat \nu_m$ and $\bar \nu_m$ of $\nu$: the idea of discretizing the Lévy density already appeared in an earlier work with P. Étoré and S. Louhichi, \cite{etore13}. The present work is also inspired by the papers \cite{cmultinomial} (for a multinomial approximation), \cite{BC04} (for passing from independent Poisson variables to independent normal random variables) and \cite{esterESAIM} (for a Bernoulli approximation). This method allows us to construct explicit Markov kernels that lead from one model to the other; these may be applied in practice to transfer minimax estimators.

The paper is organized as follows: Sections \ref{subsec:ch4parameter} and \ref{subsec:ch4experiments} are devoted to make the parameter space and the considered statistical experiments precise. The main results are given in Section \ref{subsec:ch4mainresults}, followed by Section \ref{sec:ch4experiments} in which some examples can be found. The proofs are postponed to Section \ref{sec:ch4proofs}. The paper includes an Appendix recalling the definition and some useful properties of the Le Cam distance as well as of Lévy processes.

\section{Assumptions and main results}
\subsection{The parameter space}\label{subsec:ch4parameter}
Consider a (possibly infinite) Lévy measure $\nu_0$ concentrated on  a possibly infinite interval $I\subseteq\R$, admitting a density $g>0$ with respect to Lebesgue. The parameter space of the experiments we are concerned with is a class of functions $\F=\F^{\nu_0,I}$ defined on $I$ that form a class of Lévy densities with respect to $\nu_0$: For each $f\in\F$, let $\nu$ (resp. $\hat \nu_m$) be the Lévy measure having $f$ (resp. $\hat f_m$) as a density with respect to $\nu_0$ where, for every $f\in\F$, $\hat f_m(x)$ is defined as follows.

Suppose first $x>0$. 
Given a positive integer depending on $n$, $m=m_n$, let $J_j:=(v_{j-1},v_j]$ where $v_1=\varepsilon_m\geq 0$ and $v_j$ are chosen in such a way that
\begin{equation}\label{eq:ch4Jj}
\mu_m:=\nu_0(J_j)=\frac{\nu_0\big((I\setminus[0,\varepsilon_m])\cap \R_+\big)}{m-1},\quad \forall j=2,\dots,m.
\end{equation}
In the sequel, for the sake of brevity, we will only write $m$ without making explicit the dependence on $n$. 
Define  $x_j^*:=\frac{\int_{J_j}x\nu_0(dx)}{\mu_m}$ and introduce a sequence of functions $0\leq V_j\leq \frac{1}{\mu_m}$, $j=2,\dots,m$ supported on $[x_{j-1}^*, x_{j+1}^*]$ if $j=3,\dots,m-1$, on $[\varepsilon_m, x_3^*]$ if $j=2$ and on $(I\setminus [0,x_{m-1}^*])\cap \R_+$ if $j=m$. The $V_j$'s are defined  recursively in the following way. 
\begin{itemize}
 \item $V_2$ is equal to $\frac{1}{\mu_m}$ on the interval $(\varepsilon_m, x_2^*]$ and on the interval $(x_2^*,x_3^*]$ it is chosen so that it is continuous (in particular, $V_2(x_2^*)=\frac{1}{\mu_m}$), $\int_{x_2^*}^{x_3^*}V_2(y)\nu_0(dy)=\frac{\nu_0((x_2^*, v_2])}{\mu_m}$ and $V_2(x_3^*)=0$.
 \item For $j=3,\dots,m-1$ define $V_j$ as the function $\frac{1}{\mu_m}-V_{j-1}$ on the interval $[x_{j-1}^*,x_j^*]$. On $[x_j^*,x_{j+1}^*]$ choose $V_j$ continuous and such that $\int_{x_j^*}^{x_{j+1}^*}V_j(y)\nu_0(dy)=\frac{\nu_0((x_j^*,v_j])}{\mu_m}$ and $V_j(x_{j+1}^*)=0$.
 \item Finally, let $V_m$ be the function supported on $(I\setminus [0,x_{m-1}^*]) \cap \R_+$ such that 
\begin{align*}
V_m(x)&=\frac{1}{\mu_m}-V_{m-1}(x), \quad\text{for } x \in [x_{m-1}^*,x_m^*],\\
V_m(x)&=\frac{1}{\mu_m}, \quad\text{for } x \in (I\setminus [0,x_m^*])\cap \R_+.
\end{align*}

\end{itemize}
(It is immediate to check that such a choice is always possible).
Observe that, by construction, 
\begin{equation*}
 \sum_{j=2}^m V_j(x)\mu_m=1, \quad \forall x\in (I\setminus[0,\varepsilon_m])\cap \R_+ \quad \textnormal{and} \quad \int_{(I\setminus[0,\varepsilon_m])\cap \R_+}V_j(y)\nu_0(dy)=1.
\end{equation*}

Analogously, define $\mu_m^-=\frac{\nu_0\big((I\setminus[-\varepsilon_m,0])\cap \R_-\big)}{m-1}$ and $J_{-m},\dots,J_{-2}$ such that $\nu_0(J_{-j})=\mu_m^-$ for all $j$. Then, for $x<0$, $x_{-j}^*$ is defined as $x_j^*$ by using $J_{-j}$ and $\mu_m^-$ instead of $J_j$ and $\mu_m$ and the $V_{-j}$'s are defined with the same procedure as the $V_j$'s, starting from $V_{-2}$ and proceeding by induction.

Define
\begin{equation}\label{eq:ch4hatf}
\hat f_m(x)=\I_{[-\varepsilon_m,\varepsilon_m]}(x)+\sum_{j=2}^m \bigg(V_j(x)\int_{J_j} f(y)\nu_0(dy)+V_{-j}(x)\int_{J_{-j}} f(y)\nu_0(dy)\bigg). 
\end{equation}
The definitions of the $V_j$'s above are modeled on the following example:
\begin{ex}\label{ex:Vj}
 Let $\nu_0$ be the Lebesgue measure on $[0,1]$ and $\varepsilon_m=0$. Then $v_j=\frac{j-1}{m-1}$ and $x_j^*=\frac{2j-3}{2m-2}$, $j=2,\dots,m$. The standard choice for $V_j$ (based on the construction by \cite{cmultinomial}) is given by the piecewise linear functions interpolating the values in the points $x_j^*$ specified above:
 
 \vspace{.5cm}
 \def\svgwidth{.95\textwidth}
\begingroup%
  \makeatletter%
  \providecommand\color[2][]{%
    \errmessage{(Inkscape) Color is used for the text in Inkscape, but the package 'color.sty' is not loaded}%
    \renewcommand\color[2][]{}%
  }%
  \providecommand\transparent[1]{%
    \errmessage{(Inkscape) Transparency is used (non-zero) for the text in Inkscape, but the package 'transparent.sty' is not loaded}%
    \renewcommand\transparent[1]{}%
  }%
  \providecommand\rotatebox[2]{#2}%
  \ifx\svgwidth\undefined%
    \setlength{\unitlength}{0.9\textwidth}%
    \ifx\svgscale\undefined%
      \relax%
    \else%
      \setlength{\unitlength}{\unitlength * \real{\svgscale}}%
    \fi%
  \else%
    \setlength{\unitlength}{\svgwidth}%
  \fi%
  \global\let\svgwidth\undefined%
  \global\let\svgscale\undefined%
  \makeatother%
  \begin{picture}(1,0.22125414)%
    \put(0,0){\includegraphics[width=\unitlength]{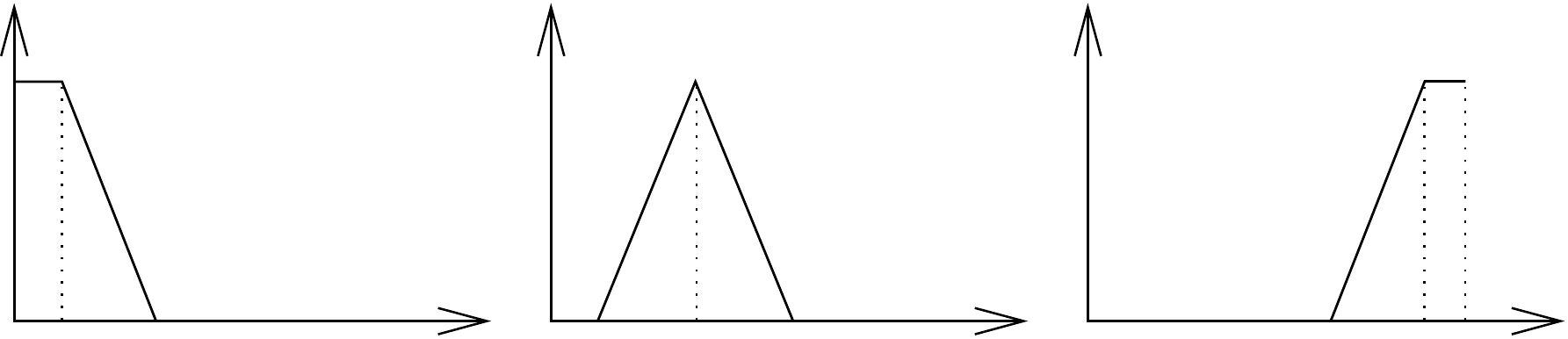}}%
    \put(0.0,-0.016){\color[rgb]{0,0,0}\makebox(0,0)[lb]{\smash{0}}}%
    \put(0.03,-0.016){\color[rgb]{0,0,0}\makebox(0,0)[lb]{\smash{$x_2^*$}}}%
    \put(0.09,-0.016){\color[rgb]{0,0,0}\makebox(0,0)[lb]{\smash{$x_3^*$}}}%
    \put(0.26,-0.016){\color[rgb]{0,0,0}\makebox(0,0)[lb]{\smash{1}}}%
    \put(-0.08,0.16633487){\color[rgb]{0,0,0}\makebox(0,0)[lb]{\smash{$m-1$}}}%
    \put(0.34,-0.016){\color[rgb]{0,0,0}\makebox(0,0)[lb]{\smash{0}}}%
    \put(0.37,-0.016){\color[rgb]{0,0,0}\makebox(0,0)[lb]{\smash{$x_{j-1}^*$}}}%
    \put(0.43,-0.016){\color[rgb]{0,0,0}\makebox(0,0)[lb]{\smash{$x_j^*$}}}%
    \put(0.49,-0.016){\color[rgb]{0,0,0}\makebox(0,0)[lb]{\smash{$x_{j+1}^*$}}}%
    \put(0.60,-0.016){\color[rgb]{0,0,0}\makebox(0,0)[lb]{\smash{1}}}%
    \put(0.26,0.16633482){\color[rgb]{0,0,0}\makebox(0,0)[lb]{\smash{$m-1$}}}%
    \put(0.68,-0.016){\color[rgb]{0,0,0}\makebox(0,0)[lb]{\smash{0}}}%
    \put(0.60,0.16633455){\color[rgb]{0,0,0}\makebox(0,0)[lb]{\smash{$m-1$}}}%
    \put(0.90,-0.016){\color[rgb]{0,0,0}\makebox(0,0)[lb]{\smash{$x_m^*$}}}%
    \put(0.82,-0.016){\color[rgb]{0,0,0}\makebox(0,0)[lb]{\smash{$x_{m-1}^*$}}}%
    \put(0.94,-0.016){\color[rgb]{0,0,0}\makebox(0,0)[lb]{\smash{1}}}%
    \put(0.15100054,0.18667777){\color[rgb]{0,0,0}\makebox(0,0)[lb]{\smash{$V_2$}}}%
    \put(0.48366954,0.18667777){\color[rgb]{0,0,0}\makebox(0,0)[lb]{\smash{$V_j$}}}%
    \put(0.81633854,0.18667777){\color[rgb]{0,0,0}\makebox(0,0)[lb]{\smash{$V_m$}}}%
  \end{picture}%
\endgroup%

 \vspace{.5cm}
 \end{ex}

\begin{remark}
The function $\hat f_m$ has been defined in such a way that the rate of convergence of the $L_2$ norm between the restriction of $f$ and $\hat f_m$ on $I\setminus[-\varepsilon_m,\varepsilon_m]$ is compatible with the rate of convergence of the other quantities appearing in the statements of Theorems \ref{ch4teo1} and \ref{ch4teo2}. For that reason, as in \cite{cmultinomial}, we have not chosen a piecewise constant approximation of $f$ but an approximation that is, at least in the simplest cases, a piecewise linear approximation of $f$. Such a choice allows us to gain an order of magnitude on the convergence rate of $\|f-\hat f_m\|_{L_2(\nu_0|{I\setminus{[-\varepsilon_m,\varepsilon_m]}})}$ at least when $\F$ is a class of sufficiently smooth functions.
\end{remark}

We now explain the assumptions we will need to make on the parameter $f \in \F = \F^{\nu_0, I}$. The superscripts $\nu_0$ and $I$ will be suppressed whenever this can lead to no confusion. We require that:
\begin{enumerate}[(H1)]
 \item  There exist constants $\kappa, M >0$ such that $\kappa\leq f(y)\leq M$, for all $y\in I$ and $f\in \F$.
\end{enumerate}
For every integer $m=m_n$, we can consider $\widehat{\sqrt{f}}_m$, the approximation of $\sqrt{f}$ constructed as $\hat f_m$ above, i.e.   $\widehat{\sqrt{f}}_m(x)=\displaystyle{\I_{[-\varepsilon_m,\varepsilon_m]}(x)+\sum_{\substack{j=-m\dots,m\\ j\neq -1,0,1.}}V_j(x)\int_{J_j} \sqrt{f(y)}\nu_0(dy)}$, and introduce the quantities:
  \begin{align*}
   A_m^2(f)&:= \int_{I\setminus \big[-\varepsilon_m,\varepsilon_m\big]}\Big(\widehat{\sqrt {f}}_m(y)-\sqrt{f(y)}\Big)^2\nu_0(dy),\\
   B_m^2(f)&:= \sum_{\substack{j=-m\dots,m\\ j\neq -1,0,1.}}\bigg(\int_{J_j}\frac{\sqrt{f(y)}}{\sqrt{\nu_0(J_j)}}\nu_0(dy)-\sqrt{\nu(J_j)}\bigg)^2,\\
   C_m^2(f)&:= \int_{-\varepsilon_m}^{\varepsilon_m}\big(\sqrt{f(t)}-1\big)^2\nu_0(dt).
  \end{align*}
The conditions defining the parameter space $\F$ are expressed by asking that the quantities introduced above converge quickly enough to zero. To state the assumptions of Theorem \ref{ch4teo1} precisely, we will assume the existence of sequences of discretizations $m = m_n\to\infty$, of positive numbers $\varepsilon_m=\varepsilon_{m_n}\to 0$ and of functions $V_j$, $j = \pm 2, \dots, \pm m$, such that:
\begin{enumerate}[(C1)]
 \item\label{cond:ch4hellinger}  $\lim\limits_{n \to \infty}n\Delta_n\sup\limits_{f \in\F}\displaystyle{\int_{I\setminus(-\varepsilon_m,\varepsilon_m)}}\Big(f(x)-\hat f_m(x)\Big)^2 \nu_0(dx) = 0$.
 \item \label{cond:ch4ABC}$\lim\limits_{n \to \infty}n\Delta_n\sup\limits_{f \in\F} \big(A_m^2(f)+B_m^2(f)+C_m^2(f)\big)=0$. 
\end{enumerate}  
Remark in particular that Condition (C\ref{cond:ch4ABC}) implies the following:
\begin{enumerate}[(H1)]
 \setcounter{enumi}{1}
 \item  $\displaystyle \sup_{f\in\F}\int_I (\sqrt{f(y)}-1)^2 \nu_0(dy) \leq L,$
\end{enumerate}
where $L = \sup_{f \in \F} \int_{-\varepsilon_m}^{\varepsilon_m} (\sqrt{f(x)}-1)^2\nu_0(dx) + (\sqrt{M}+1)^2\nu_0\big(I\setminus (-\varepsilon_m, \varepsilon_m)\big)$, for any choice of $m$ such that the quantity in the limit appearing in Condition (C\ref{cond:ch4ABC}) is finite.

Theorem \ref{ch4teo2} has slightly stronger hypotheses, defining possibly smaller parameter spaces: We will assume the existence of sequences $m_n$, $\varepsilon_m$ and $V_j$, $j = \pm 2, \dots, \pm m$ (possibly different from the ones above) such that Condition (C1) is verified and the following stronger version of Condition (C2) holds:
\begin{enumerate}[(C1')]
 \setcounter{enumi}{1}
\item$\lim\limits_{n \to \infty}n\Delta_n\sup\limits_{f \in\F} \big(A_m^2(f)+B_m^2(f)+nC_m^2(f)\big)=0$. 
\end{enumerate}

Finally, some of our results have a more explicit statement under the hypothesis of finite variation which we state as:
\begin{itemize}
\item[(FV)] $\int_I (|x|\wedge 1)\nu_0(dx)<\infty$.
\end{itemize}
\begin{remark}
 The Condition (C1) and those involving the quantities $A_m(f)$ and $B_m(f)$ all concern similar but slightly different approximations of $f$. In concrete examples, they may all be expected to have the same rate of convergence but to keep the greatest generality we preferred to state them separately. On the other hand, conditions on the quantity $C_m(f)$ are purely local around zero, requiring the parameters $f$ to converge quickly enough to 1.
\end{remark}

\begin{exs}\label{ex:ch4esempi}
 To get a grasp on Conditions (C1), (C2) we analyze here three different examples according to the different behavior of $\nu_0$ near $0\in I$. In all of these cases the parameter space $\F^{\nu_0, I}$ will be a subclass of $\F_{(\gamma,K,\kappa,M)}^I$ defined as in \eqref{ch4:fholder}. Recall that the conditions (C1), (C2) and (C2') depend on the choice of sequences $m_n$, $\varepsilon_m$ and functions $V_j$. For the first two of the three examples, where $I = [0,1]$, we will make the standard choice for $V_j$ of triangular and trapezoidal functions, similarly to those in Example \ref{ex:Vj}. Namely, for $j = 3, \dots, m-1$ we have
 \begin{equation}\label{eq:ch4vj}
 V_j(x) = \I_{(x_{j-1}^*, x_j^*]}(x) \frac{x-x_{j-1}^*}{x_j^*-x_{j-1}^*} \frac{1}{\mu_m} + \I_{(x_{j}^*, x_{j+1}^*]}(x) \frac{x_{j+1}^*-x}{x_{j+1}^*-x_{j}^*} \frac{1}{\mu_m};
 \end{equation}
 the two extremal functions $V_2$ and $V_m$ are chosen so that $V_2 \equiv \frac{1}{\mu_m}$ on $(\varepsilon_m, x_2^*]$ and $V_m \equiv \frac{1}{\mu_m}$ on $(x_m^*, 1]$.  In the second example, where $\nu_0$ is infinite, one is forced to take $\varepsilon_m > 0$ and to keep in mind that the $x_j^*$ are not uniformly distributed on $[\varepsilon_m,1]$. Proofs of all the statements here can be found in Section \ref{subsec:esempi}.
 
 \textbf{1. The finite case:} $\nu_0\equiv \Leb([0,1])$.
  
  In this case we are free to choose $\F^{\Leb, [0,1]} = \F_{(\gamma, K, \kappa, M)}^{[0,1]}$. Indeed, as $\nu_0$ is finite, there is no need to single out the first interval $J_1=[0,\varepsilon_m]$, so that $C_m(f)$ does not enter in the proofs and the definitions of $A_m(f)$ and $B_m(f)$ involve integrals on the whole of $[0,1]$. Also, the choice of the $V_j$'s as in \eqref{eq:ch4vj} guarantees that $\int_0^1 V_j(x) dx = 1$. Then, the quantities $\|f-\hat f_m\|_{L_2([0,1])}$, $A_m(f)$ and $B_m(f)$ all have the same rate of convergence, which is given by:
  \begin{equation*}
 \sqrt{\int_0^1\Big(f(x)-\hat f_m(x)\Big)^2 \nu_0(dx)}+A_m(f)+B_m(f)=O\Big(m^{-\gamma-1}+m^{-\frac{3}{2}}\Big), 
  \end{equation*}
  uniformly on $f$. See Section \ref{subsec:esempi} for a proof.

  \textbf{2. The finite variation case:} $\frac{d\nu_0}{d\Leb}(x)=x^{-1}\I_{[0,1]}(x)$.
  
  In this case, the parameter space $\F^{\nu_0, [0,1]}$ is a proper subset of $\F_{(\gamma, K, \kappa, M)}^{[0,1]}$. Indeed, as we are obliged to choose $\varepsilon_m > 0$, we also need to impose that $C_m(f) = o\big(\frac{1}{n\sqrt{\Delta_n}}\big)$, with uniform constants with respect to $f$, that is, that all $f \in \F$ converge to 1 quickly enough as $x \to 0$. Choosing $\varepsilon_m = m^{-1-\alpha}$, $\alpha> 0$ we have that $\mu_m=\frac{\ln (\varepsilon_m^{-1})}{m-1}$, $v_j =\varepsilon_m^{\frac{m-j}{m-1}}$ and $x_j^* =\frac{(v_{j}-v_{j-1})}{\mu_m}$. In particular,  $\max_j|v_{j-1}-v_j|=|v_m-v_{m-1}|=O\Big(\frac{\ln m}{m}\Big)$. 
  Also in this case one can prove that the standard choice of $V_j$ described above leads to $\int_{\varepsilon_m}^1 V_j(x) \frac{dx}{x} = 1$. Again, the quantities $\|f-\hat f_m\|_{L_2(\nu_0|{I\setminus{[0,\varepsilon_m]}})}$, $A_m(f)$ and $B_m(f)$ have the same rate of convergence given by: 
  \begin{equation}\label{eq:ch4ex2}
\sqrt{\int_{\varepsilon_m}^1\Big(f(x)-\hat f_m(x)\Big)^2 \nu_0(dx)} +A_m(f)+B_m(f)=O\bigg(\bigg(\frac{\ln m}{m}\bigg)^{\gamma+1} \sqrt{\ln (\varepsilon_m^{-1})}\bigg),   
  \end{equation}
 uniformly on $f$. The condition on $C_m(f)$ depends on the behavior of $f$ near $0$. For example, it is ensured if one considers a parametric family of the form $f(x)=e^{-\lambda x}$ with a bounded $\lambda > 0$.
See Section \ref{subsec:esempi} for a proof.

 \textbf{3. The infinite variation, non-compactly supported case:} $\frac{d\nu_0}{d\Leb}(x)=x^{-2}\I_{\R_+}(x)$.
  
  This example involves significantly more computations than the preceding ones, since the classical triangular choice for the functions $V_j$ would not have integral equal to 1 (with respect to $\nu_0$), and the support is not compact. The parameter space $\F^{\nu_0, [0, \infty)}$ can still be chosen as a proper subclass of $\F_{(\gamma, K, \kappa, M)}^{[0,\infty)}$, again by imposing that $C_m(f)$ converges to zero quickly enough (more details about this condition are discussed in Example \ref{ex3}). We divide the interval $[0, \infty)$ in $m$ intervals $J_j = [v_{j-1}, v_j)$ with:
  $$
  v_0 = 0; \quad v_1 = \varepsilon_m; \quad v_j = \frac{\varepsilon_m(m-1)}{m-j};\quad v_m = \infty; \quad \mu_m = \frac{1}{\varepsilon_m(m-1)}.
  $$
  To deal with the non-compactness problem, we choose some ``horizon'' $H(m)$ that goes to infinity slowly enough as $m$ goes to infinity and we bound the $L_2$ distance between $f$ and $\hat f_m$ for $x > H(m)$ by $2\sup\limits_{x\geq H(m)}\frac{f(x)^2}{H(m)}$. 
  We have:
  $$\|f-\hat f_m\|_{L_2(\nu_0|{I\setminus{[0,\varepsilon_m]}})}^2+A_m^2(f)+B_m^2(f)=O\bigg(\frac{H(m)^{3+4\gamma}}{(\varepsilon_m m)^{2+2\gamma}}+\sup_{x\geq H(m)}\frac{f(x)^2}{H(m)}\bigg).$$
  In the general case where the best estimate for $\displaystyle{\sup_{x\geq H(m)}f(x)^2}$ is simply given by $M^2$, an optimal choice for $H(m)$ is $\sqrt{\varepsilon_m m}$, that gives a rate of convergence:
  $$
  \|f-\hat f_m\|_{L_2(\nu_0|{I\setminus{[0,\varepsilon_m]}})}^2+A_m^2(f)+B_m^2(f) =O\bigg( \frac{1}{\sqrt{\varepsilon_m m}}\bigg),
  $$ 
  independently of $\gamma$. See Section \ref{subsec:esempi} for a proof.
  \end{exs} 

\subsection{Definition of the experiments}\label{subsec:ch4experiments}

Let $(x_t)_{t\geq 0}$ be the canonical process on the Skorokhod space $(D,\D)$ and denote by $P^{(b,0,\nu)}$ the law induced on $(D,\D)$ by a Lévy process with characteristic triplet $(b,0,\nu)$. We will write $P_t^{(b,0,\nu)}$ for the restriction of $P^{(b,0,\nu)}$ to the $\sigma$-algebra $\D_t$ generated by $\{x_s:0\leq s\leq t\}$ (see \ref{sec:ch4levy} for the precise definitions). Let $Q_t^{(b,0,\nu)}$ be the marginal law at time $t$ of a Lévy process with characteristic triplet ${(b,0,\nu)}$. In the case where $\int_{|y|\leq 1}|y|\nu(dy)<\infty$ we introduce the notation $\gamma^{\nu}:=\int_{|y|\leq 1}y\nu(dy)$; then, Condition (H2) guarantees the finiteness of $\gamma^{\nu-\nu_0}$ (see Remark 33.3 in \cite{sato} for more details).

Recall that we introduced the discretization $t_i=T_n\frac{i}{n}$ of $[0,T_n]$ and denote by $\textbf Q_n^{(\gamma^{\nu-\nu_0},0,\nu)}$ the laws of the $n+1$ marginals of $(x_t)_{t\geq 0}$ at times $t_i$, $i=0,\dots,n$. We will consider the following statistical models, depending on a fixed, possibly infinite, Lévy measure $\nu_0$ concentrated on $I$ (clearly, the models with the subscript $FV$ are meaningful only under the assumption (FV)):
\begin{align*}
 \mo_{n,FV}^{\nu_0}&=\bigg(D,\D_{T_n},\Big\{P_{T_n}^{(\gamma^{\nu},0,\nu)}:f:=\frac{d\nu}{d\nu_0}\in\F^{\nu_0,I}\Big\}\bigg),\\
\Qi_{n,FV}^{\nu_0}&=\bigg(\R^{n+1},\B(\R^{n+1}),\Big\{ \textbf Q_{n}^{(\gamma^{\nu},0,\nu)}:f:=\frac{d\nu}{d\nu_0}\in\F^{\nu_0,I}\Big\}\bigg),\\
 \mo_{n}^{\nu_0}&=\bigg(D,\D_{T_n},\Big\{P_{T_n}^{(\gamma^{\nu-\nu_0},0,\nu)}:f:=\frac{d\nu}{d\nu_0}\in\F^{\nu_0,I}\Big\}\bigg),\\
 \Qi_{n}^{\nu_0}&=\bigg(\R^{n+1},\B(\R^{n+1}),\Big\{\textbf Q_{n}^{(\gamma^{\nu-\nu_0},0,\nu)}:f:=\frac{d\nu}{d\nu_0}\in\F^{\nu_0,I}\Big\}\bigg).
  \end{align*}
Finally, let us introduce the Gaussian white noise model that will appear in the statement of our main results. For that, let us denote by $(C(I),\Ci)$ the space of continuous mappings from $I$ into $\R$ endowed with its standard filtration, by $g$ the density of $\nu_0$ with respect to the Lebesgue measure.
We will require $g>0$ and let $\mathbb W_n^f$ be the law induced on $(C(I),\Ci)$ by the stochastic process satisfying:
\begin{align}\label{eqn:ch4Wf}
 dy_t=\sqrt{f(t)}dt+\frac{dW_t}{2\sqrt{T_n}\sqrt{g(t)}}, \quad  t\in I,
\end{align}
where $(W_t)_{t\in\R}$ denotes a Brownian motion on $\R$ with $W_0=0$.
Then we set:
\begin{equation*}
 \Wh_n^{\nu_0}=\Big(C(I),\Ci,\{\mathbb W_n^{f}:f\in\F^{\nu_0,I}\}\Big).
\end{equation*}
Observe that when $\nu_0$ is a finite Lévy measure, then $\Wh_n^{\nu_0}$ is equivalent to the statistical model associated with the continuous observation of a process $(\tilde y_t)_{t\in I}$ defined by:
\begin{align*}
 d\tilde y_t=\sqrt{f(t)g(t)}dt+\frac{d W_t}{2\sqrt{T_n}}, \quad  t\in I.
\end{align*}

\subsection{Main results}\label{subsec:ch4mainresults}
Using the notation introduced in Section \ref{subsec:ch4parameter}, we now state our main results. For brevity of notation, we will denote by $H(f,\hat f_m)$ (resp. $L_2(f,\hat f_m)$) the Hellinger distance (resp. the $L_2$ distance) between the Lévy measures $\nu$ and $\hat\nu_m$ restricted to $I\setminus{[-\varepsilon_m,\varepsilon_m]}$, i.e.:
\begin{align*}
H^2(f,\hat f_m)&:=\int_{I\setminus{[-\varepsilon_m,\varepsilon_m]}}\Big(\sqrt{f(x)}-\sqrt{\hat f_m(x)}\Big)^2 \nu_0(dx),\\
L_2(f,\hat f_m)^2&:=\int_{I\setminus{[-\varepsilon_m,\varepsilon_m]}}\big(f(y)-\hat f_m(y)\big)^2\nu_0(dy).
\end{align*}
Observe that Condition (H1) implies (see Lemma \ref{lemma:ch4hellinger})
$$\frac{1}{4M}L_2(f,\hat f_m)^2\leq H^2(f,\hat f_m)\leq \frac{1}{4\kappa}L_2(f,\hat f_m)^2.$$

\begin{theorem}\label{ch4teo1}
Let $\nu_0$ be a known Lévy measure concentrated on a (possibly infinite) interval $I\subseteq \R$ and having strictly positive density with respect to the Lebesgue measure. Let us choose a parameter space $\F^{\nu_0, I}$ such that there exist a sequence $m = m_n$ of integers, functions $V_j$, $j = \pm 2, \dots, \pm m$ and a sequence $\varepsilon_m \to 0$ as $m \to \infty$ such that Conditions \textnormal{(H1), (C1), (C2)} are satisfied for $\F = \F^{\nu_0, I}$. Then, for $n$ big enough we have:  
\begin{align}
\Delta(\mo_n^{\nu_0}, \Wh_n^{\nu_0}) &= O\bigg(\sqrt{n\Delta_n}\sup_{f\in \F}\Big(A_m(f)+B_m(f)+C_m(f)\Big)\bigg) \nonumber \\
                            & +O\bigg(\sqrt{n\Delta_n}\sup_{f\in\F}L_2(f, \hat f_m)+\sqrt{\frac{m}{n\Delta_n}\Big(\frac{1}{\mu_m}+\frac{1}{\mu_m^-}\Big)}\bigg). \label{eq:teo1}
\end{align}
\end{theorem}

\begin{theorem}\label{ch4teo2}
Let $\nu_0$ be a known Lévy measure concentrated on a (possibly infinite) interval $I\subseteq \R$ and having strictly positive density with respect to the Lebesgue measure. Let us choose a parameter space $\F^{\nu_0, I}$ such that there exist a sequence $m = m_n$ of integers, functions $V_j$, $j = \pm 2, \dots, \pm m$ and a sequence $\varepsilon_m \to 0$ as $m \to \infty$ such that Conditions \textnormal{(H1), (C1), (C2')} are satisfied for $\F = \F^{\nu_0, I}$. Then, for $n$ big enough we have:
\begin{align} 
 \Delta(\Qi_n^{\nu_0}, \Wh_n^{\nu_0})& = O\bigg( \nu_0\Big(I\setminus[-\varepsilon_m,\varepsilon_m]\Big)\sqrt{n\Delta_n^2}+\frac{m\ln m}{\sqrt{n}}+\sqrt{n\sqrt{\Delta_n}\sup_{f\in\F}C_m(f)}\bigg) \nonumber \\
 &+O\bigg(\sqrt{n\Delta_n}\sup_{f\in\F}\Big(A_m(f)+B_m(f)+H(f,\hat f_m)\Big)\bigg).\label{eq:teo2}
\end{align}

\end{theorem}

\begin{corollary}\label{cor:ch4generale}
 Let $\nu_0$ be as above and let us choose a parameter space $\F^{\nu_0, I}$ so that there exist sequences $m_n'$, $\varepsilon_m'$, $V_j'$ and $m_n''$, $\varepsilon_m''$, $V_j''$ such that:
 \begin{itemize}
  \item Conditions (H1), (C1) and (C2) hold for $m_n'$, $\varepsilon_m'$, $V_j'$, and $\frac{m'}{n\Delta_n}\Big(\frac{1}{\mu_{m'}}+\frac{1}{\mu_{m'}^-}\Big)$ tends to zero.
  \item Conditions (H1), (C1) and (C2') hold for $m_n''$, $\varepsilon_m''$, $V_j''$, and $\nu_0\Big(I\setminus[-\varepsilon_{m''},\varepsilon_{m''}]\Big)\sqrt{n\Delta_n^2}+\frac{m''\ln m''}{\sqrt{n}}$ tends to zero.
 \end{itemize}
  Then the statistical models $\mo_{n}^{\nu_0}$ and $\Qi_{n}^{\nu_0}$ are asymptotically equivalent:
 $$\lim_{n\to\infty}\Delta(\mo_{n}^{\nu_0},\Qi_{n}^{\nu_0})=0,$$
 \end{corollary}

If, in addition, the Lévy measures have finite variation, i.e. if we assume (FV), then the same results hold replacing $\mo_{n}^{\nu_0}$ and $\Qi_{n}^{\nu_0}$ by $\mo_{n,FV}^{\nu_0}$ and $\Qi_{n,FV}^{\nu_0}$, respectively (see Lemma \ref{ch4LC}).

\section{Examples}\label{sec:ch4experiments}

We will now analyze three different examples, underlining the different behaviors of the Lévy measure $\nu_0$ (respectively, finite, infinite with finite variation and infinite with infinite variation). The three chosen Lévy measures are $\I_{[0,1]}(x) dx$, $\I_{[0,1]}(x) \frac{dx}{x}$ and $\I_{\R_+}(x)\frac{dx}{x^2}$. In all three cases we assume the parameter $f$ to be uniformly bounded and with uniformly $\gamma$-Hölder derivatives: We will describe adequate subclasses $\F^{\nu_0, I} \subseteq \F_{(\gamma, K, \kappa, M)}^I$ defined as in \eqref{ch4:fholder}. It seems very likely that the same results that are highlighted in these examples hold true for more general Lévy measures; however, we limit ourselves to these examples in order to be able to explicitly compute the quantities involved ($v_j$, $x_j^*$, etc.) and hence estimate the distance between $f$ and $\hat f_m$ as in Examples \ref{ex:ch4esempi}.

In the first of the three examples, where $\nu_0$ is the Lebesgue measure on $I=[0,1]$, we are considering the statistical models associated with the discrete and continuous observation of a compound Poisson process with Lévy density $f$. Observe that $\Wh_n^{\Leb}$ reduces to the statistical model associated with the continuous observation of a trajectory from:
 $$dy_t=\sqrt{f(t)}dt+\frac{1}{2\sqrt{T_n}}dW_t,\quad t\in [0,1].$$
In this case we have:
\begin{ex}\label{ex:ch4CPP}(Finite Lévy measure).
Let $\nu_0$ be the Lebesgue measure on $I=[0,1]$ and let $\F = \F^{\Leb, [0,1]}$ be any subclass of $\F_{(\gamma, K, \kappa, M)}^{[0,1]}$ for some strictly positive constants $K$, $\kappa$, $M$ and $\gamma\in(0,1]$. Then:
 \begin{equation*}
 \lim_{n\to\infty}\Delta(\mo_{n,FV}^{\Leb},\Wh_n^{\Leb})=0 \ \textnormal{ and } \ \lim_{n\to\infty}\Delta(\Qi_{n,FV}^{\Leb},\Wh_n^{\Leb})=0.
 \end{equation*}
More precisely, 
$$\Delta(\mo_{n,FV}^{\Leb},\Wh_n^{\Leb})=\begin{cases}O\Big((n\Delta_n)^{-\frac{\gamma}{4+2\gamma}}\Big)\quad \textnormal{if } \ \gamma\in\big(0,\frac{1}{2}\big],\\
 O\Big((n \Delta_n)^{-\frac{1}{10}}\Big)\quad \textnormal{if } \ \gamma\in\big(\frac{1}{2},1\big].
                                         \end{cases}
$$
 In the case where $\Delta_n = n^{-\beta}$, $\frac{1}{2} < \beta < 1$, an upper bound for the rate of convergence of $\Delta(\Qi_{n,FV}^{\Leb}, \Wh_n^{\Leb})$ is 
 $$\Delta(\Qi_{n,FV}^{\Leb}, \Wh_n^{\Leb})=\begin{cases}
 O\Big(n^{-\frac{\gamma+\beta}{4+2\gamma}}\ln n\Big)\quad \textnormal{if } \ \gamma\in\big(0,\frac{1}{2}\big) \text{ and }\frac{2+2\gamma}{3+2\gamma} \leq \beta < 1,\\
 O\Big(n^{\frac{1}{2}-\beta}\ln n\Big)\quad \textnormal{if } \ \gamma\in\big(0,\frac{1}{2}\big) \text{ and } \frac{1}{2} < \beta < \frac{2+2\gamma}{3+2\gamma},\\
 O\Big(n^{-\frac{2\beta+1}{10}}\ln n\Big)\quad \textnormal{if } \ \gamma\in\big[\frac{1}{2},1\big] \text{ and } \frac{3}{4} \leq \beta < 1,\\
 O\Big(n^{\frac{1}{2}-\beta}\ln n\Big)\quad \textnormal{if } \ \gamma\in\big[\frac{1}{2},1\big] \text{ and } \frac{1}{2} < \beta < \frac{3}{4}.
 \end{cases}$$
 See Section \ref{subsec:ch4ex1} for a proof.
\end{ex}

\begin{ex}\label{ch4ex2}(Infinite Lévy measure with finite variation).
 Let $X$ be a truncated Gamma process with (infinite) Lévy measure of the form:
 $$\nu(A)=\int_A \frac{e^{-\lambda x}}{x}dx,\quad A\in\B([0,1]).$$
 Here $\F^{\nu_0, I}$ is a 1-dimensional parametric family in $\lambda$, assuming that there exists a known constant $\lambda_0$ such that $0<\lambda\leq \lambda_0<\infty$, $f(t) = e^{-\lambda t}$ and $d\nu_0(x)=\frac{1}{x}dx$. In particular, the $f$ are Lipschitz, i.e. $\F^{\nu_0, [0,1]} \subset \F_{(\gamma = 1, K, \kappa, M)}^{[0,1]}$.
 The discrete or continuous observation (up to time $T_n$) of $X$ are asymptotically equivalent to $\Wh_n^{\nu_0}$, the statistical model associated with the observation of a trajectory of the process $(y_t)$:
  $$dy_t=\sqrt{f(t)}dt+\frac{\sqrt tdW_t}{2\sqrt{T_n}},\quad t\in[0,1].$$
 More precisely, in the case where $\Delta_n = n^{-\beta}$, $\frac{1}{2} < \beta < 1$, an upper bound for the rate of convergence of $\Delta(\Qi_{n,FV}^{\nu_0}, \Wh_n^{\nu_0})$ is
 \begin{equation*}
 \Delta(\Qi_{n,FV}^{\nu_0},\Wh_n^{\nu_0}) = \begin{cases}
                                             O\big(n^{\frac{1}{2}-\beta} \ln n\big) & \text{if } \frac{1}{2} < \beta \leq \frac{9}{10}\\
                                             O\big(n^{-\frac{1+2\beta}{7}} \ln n\big) & \text{if } \frac{9}{10} < \beta < 1.
                                            \end{cases}
 \end{equation*}
 Concerning the continuous setting we have:
 $$\Delta(\mo_{n,FV}^{\nu_0},\Wh_n^{\nu_0})=O\Big(n^{\frac{\beta-1}{6}} \big(\ln n\big)^{\frac{5}{2}}\Big) = O\Big(T_n^{-\frac{1}{6}} \big(\ln T_n\big)^\frac{5}{2}\Big).$$
 See Section \ref{subsec:ch4ex2} for a proof.
\end{ex}
\begin{ex}\label{ex3}(Infinite Lévy measure, infinite variation).
 Let $X$ be a pure jump Lévy process with infinite Lévy measure of the form:
 $$\nu(A)=\int_A \frac{2-e^{-\lambda x^3}}{x^2}dx,\quad A\in\B(\R^+).$$
 Again, we are considering a parametric family in $\lambda > 0$, assuming that the parameter stays bounded below a known constant $\lambda_0$. Here, $f(t) =2- e^{-\lambda t^3}$, hence $1\leq f(t)\leq 2$, for all $t\geq 0$, and $f$ is Lipschitz, i.e. $\F^{\nu_0, \R_+} \subset \F_{(\gamma = 1, K, \kappa, M)}^{\R_+}$. The discrete or continuous observations (up to time $T_n$) of $X$ are asymptotically equivalent to the statistical model associated with the observation of a trajectory of the process $(y_t)$:
  $$dy_t=\sqrt{f(t)}dt+\frac{tdW_t}{2\sqrt{T_n}},\quad t\geq 0.$$
 More precisely, in the case where $\Delta_n = n^{-\beta}$, $0 < \beta < 1$, an upper bound for the rate of convergence of $\Delta(\Qi_{n}^{\nu_0}, \Wh_n^{\nu_0})$ is
    $$
  \Delta(\Qi_{n}^{\nu_0},\Wh_n^{\nu_0}) = \begin{cases}
                                              O\big(n^{\frac{1}{2} - \frac{2}{3}\beta}\big)& \text{if } \frac{3}{4} < \beta < \frac{12}{13}\\
                                              O\big(n^{-\frac{1}{6}+\frac{\beta}{18}} (\ln n)^{\frac{7}{6}}\big) &\text{if }  \frac{12}{13}\leq \beta<1.
                                             \end{cases}
  $$
  
  In the continuous setting, we have
  $$
  \Delta(\mo_{n}^{\nu_0},\Wh_n^{\nu_0})=O\big(n^{\frac{3\beta-3}{34}}(\ln n)^{\frac{7}{6}}\big) = O\big(T_n^{-\frac{3}{34}} (\ln T_n)^{\frac{7}{6}}\big).
  $$
   See Section \ref{subsec:ch4ex3} for a proof.
\end{ex}

\section{Proofs of the main results}\label{sec:ch4proofs}
In order to simplify notations, the proofs will be presented in the case $I\subseteq \R^+$. Nevertheless, this allows us to present all the main difficulties, since they can only appear near 0. To prove Theorems \ref{ch4teo1} and \ref{ch4teo2} we need to introduce several intermediate statistical models. In that regard, let us denote by $Q_j^f$ the law of a Poisson random variable with mean $T_n\nu(J_j)$ (see \eqref{eq:ch4Jj} for the definition of $J_{j}$). 
We will denote by $\mathscr{L}_m$ the statistical model associated with the family of probabilities $\big\{\bigotimes_{j=2}^m Q_j^f:f\in\F \big\}$:
\begin{equation}\label{eq:ch4l}
 \mathscr{L}_m=\bigg(\bar{\N}^{m-1},\mathcal P(\bar{\N}^{m-1}), \bigg\{\bigotimes_{j=2}^m Q_j^f:f\in\F \bigg\}\bigg).
\end{equation}

By $N_{j}^f$ we mean the law of a Gaussian random variable $\Nn(2\sqrt{T_n\nu(J_j)},1)$ and by $\mathscr{N}_m$ the statistical model associated with the family of probabilities $\big\{\bigotimes_{j=2}^m N_j^f:f\in\F \big\}$:
\begin{equation}\label{eq:ch4n}
 \mathscr{N}_m=\bigg(\R^{m-1},\mathscr B(\R^{m-1}), \bigg\{\bigotimes_{j=2}^m N_j^f:f\in\F \bigg\}\bigg).
\end{equation}

For each $f\in\F$, let $\bar \nu_m$ be the measure having $\bar f_m$ as a density with respect to $\nu_0$ where, for every $f\in\F$, $\bar f_m$ is defined as follows.
\begin{equation}\label{eq:ch4barf}
\bar f_m(x):=
 \begin{cases}
\quad 1 & \textnormal{if } x\in J_1,\\
\frac{\nu(J_j)}{{\nu_0}(J_{j})} & \textnormal{if } x\in J_{j}, \quad j = 2,\dots,m.
 \end{cases}
\end{equation}
Furthermore, define
\begin{equation}\label{eq:ch4modellobar}
\bar\mo_{n}^{\nu_0}=\bigg(D,\D_{T_n},\Big\{P_{T_n}^{(\gamma^{\bar \nu_m-\nu_0},0,\bar\nu_m)}:\frac{d\bar\nu_m}{d\nu_0}\in\F\Big\}\bigg). 
\end{equation}

\subsection{Proof of Theorem \ref{ch4teo1}}

We begin by a series of lemmas that will be needed in the proof. Before doing so, let us underline the scheme of the proof. We recall that the goal is to prove that estimating $f=\frac{d\nu}{d\nu_0}$ from the continuous observation of a Lévy process $(X_t)_{t\in[0,T_n]}$ without Gaussian part and having Lévy measure $\nu$ is asymptotically equivalent to estimating $f$ from the Gaussian white noise model:

   $$dy_t=\sqrt{f(t)}dt+\frac{1}{2\sqrt{T_n g(t)}}dW_t,\quad g=\frac{d\nu_0}{d\Leb},\quad t\in I.$$

 Also, recall the definition of $\hat \nu_m$ given in \eqref{eq:ch4hatf} and read $\mo_1 \overset{\Delta} \Longleftrightarrow \mo_2$ as $\mo_1$ is asymptotically equivalent to $\mo_2$. Then, we can outline the proof in the following way.

 \begin{itemize}
  \item Step 1: $P_{T_n}^{(\gamma^{\nu-\nu_0},0,\nu)} \overset{\Delta} \Longleftrightarrow P_{T_n}^{(\gamma^{\hat\nu_m-\nu_0},0,\hat\nu_m)}$;
   
   \item Step 2: $P_{T_n}^{(\gamma^{\hat\nu_m-\nu_0},0,\hat\nu_m)} \overset{\Delta} \Longleftrightarrow \bigotimes_{j=2}^m \mo(T_n\nu(J_j))$ (Poisson approximation). 
   
   Here $\bigotimes_{j=2}^m \mo(T_n\nu(J_j))$ represents a statistical model associated with the observation of $m-1$ independent Poisson r.v. of parameters $T_n\nu(J_j)$;
   
   \item  Step 3: $\bigotimes_{j=2}^m \mo(T_n \nu(J_j)) \overset{\Delta} \Longleftrightarrow \bigotimes_{j=2}^m \Nn(2\sqrt{T_n\nu(J_j)},1)$ (Gaussian approximation);
  \item  Step 4: $\bigotimes_{j=2}^m \Nn(2\sqrt{T_n\nu(J_j)},1)\overset{\Delta} \Longleftrightarrow (y_t)_{t\in I}$. 
 \end{itemize}

 Lemmas \ref{lemma:ch4poisson}--\ref{lemma:ch4kernel}, below, are the key ingredients of Step 2. 
\begin{lemma}\label{lemma:ch4poisson}
Let $\bar\mo_{n}^{\nu_0}$ and $\mathscr{L}_m$ be the statistical models defined in \eqref{eq:ch4modellobar} and \eqref{eq:ch4l}, respectively. Under the Assumption (H2) we have:
$$\Delta(\bar\mo_{n}^{\nu_0}, \mathscr{L}_m)=0, \textnormal{ for all } m.$$ 
\end{lemma}
\begin{proof}
 Denote by $\bar \N=\N\cup \{\infty\}$ and consider the statistics $S:(D,\D_{T_n})\to \big(\bar\N^{m-1},\mathcal{P}(\bar\N^{m-1})\big)$  defined by 
\begin{equation}\label{eq:ch4S}
S(x)=\Big(N_{T_n}^{x;\,2},\dots,N_{T_n}^{x;\,m}\bigg)\quad \textnormal{with} \quad
 N_{T_n}^{x;\,j}=\sum_{r\leq T_n}\I_{J_{j}}(\Delta x_r).
\end{equation}
An application of Theorem \ref{ch4teosato} to $P_{T_n}^{(\gamma^{\bar \nu_m-\nu_0},0,\bar \nu_m)}$ and $P_{T_n}^{(0,0,\nu_0)}$, yields 
\begin{equation*}
\frac{d P_{T_n}^{(\gamma^{\bar \nu_m-\nu_0},0,\bar \nu_m)}}{dP_{T_n}^{(0,0,\nu_0)}}(x)=\exp\bigg(\sum_{j=2}^m \bigg(\ln\Big(\frac{\nu(J_j)}{\nu_0(J_j)}\Big)\bigg) N_{T_n}^{x;j}-T_n\int_I(\bar f_m(y)-1)\nu_0(dy)\bigg).
\end{equation*}
Hence, by means of the Fisher factorization theorem, we conclude that $S$ is a sufficient statistics for $\bar\mo_{n}^{\nu_0}$. Furthermore, under $P_{T_n}^{(\gamma^{\bar \nu_m-\nu_0},0,\bar \nu_m)}$, the random variables $N_{T_n}^{x;j}$ have Poisson distributions $Q_{j}^f$ with means $T_n\nu(J_j)$.  Then, by means of Property \ref{ch4fatto3}, we get $\Delta(\bar\mo_{n}^{\nu_0}, \mathscr{L}_m)=0, \textnormal{ for all } m.$
\end{proof}

Let us denote by $\hat Q_j^f$ the law of a Poisson random variable with mean $T_n\int_{J_j}\hat f_m(y)\nu_0(dy)$ 
and let $\hat{\mathscr{L}}_m$ be the statistical model associated with the family of probabilities $\{\bigotimes_{j=2}^m \hat Q_j^f:f\in \F\}$.

\begin{lemma}\label{lemma:ch4poissonhatf}
 $$\Delta(\mathscr L_m,\hat{\mathscr{L}}_m)\leq \sup_{f\in \F}\sqrt{\frac{T_n}{\kappa}\int_{I\setminus[0,\varepsilon_m]}\big(f(y)-\hat f_m(y)\big)^2\nu_0(dy)}.$$
\end{lemma}
\begin{proof}
By means of Facts \ref{ch4h}--\ref{fact:ch4hellingerpoisson}, we get:
\begin{align*}
 \Delta(\mathscr L_m,\hat{\mathscr{L}}_m)&\leq \sup_{f\in\F}H\bigg(\bigotimes_{j=2}^m  Q_j^f,\bigotimes_{j=2}^m \hat Q_j^f\bigg)\\
                             &\leq \sup_{f\in\F}\sqrt{\sum_{j=2}^m 2 H^2(Q_j^f,\hat Q_j^f)}\\
                           & =\sup_{f\in\F}\sqrt 2\sqrt{\sum_{j=2}^m\bigg(1-\exp\bigg(-\frac{T_n}{2}\bigg[\sqrt{\int_{J_j}\hat f(y)\nu_0(dy)}-\sqrt{\int_{J_j} f(y)\nu_0(dy)}\bigg]^2\bigg)\bigg)}.
\end{align*}
By making use of the fact that $1-e^{-x}\leq x$ for all $x\geq 0$ and the equality $\sqrt a-\sqrt b= \frac{a-b}{\sqrt a+\sqrt b}$ combined with the lower bound $f\geq \kappa$ (that also implies $\hat f_m\geq \kappa$) and finally the Cauchy-Schwarz inequality, we obtain:
\begin{align*}
 &1-\exp\bigg(-\frac{T_n}{2}\bigg[\sqrt{\int_{J_j}\hat f(y)\nu_0(dy)}-\sqrt{\int_{J_j} f(y)\nu_0(dy)}\bigg]^2\bigg)\\
 &\leq \frac{T_n}{2}\bigg[\sqrt{\int_{J_j}\hat f(y)\nu_0(dy)}-\sqrt{\int_{J_j} f(y)\nu_0(dy)}\bigg]^2\\
& \leq \frac{T_n}{2} \frac{\bigg(\int_{J_j}(f(y)-\hat f_m(y))\nu_0(dy)\bigg)^2}{\kappa \nu_0(J_j)}\\
&\leq \frac{T_n}{2\kappa} \int_{J_j}\big(f(y)-\hat f_m(y)\big)^2\nu_0(dy).
 \end{align*}
Hence,
$$H\bigg(\bigotimes_{j=2}^m  Q_j^f,\bigotimes_{j=2}^m \hat Q_j^f\bigg)\leq \sqrt{\frac{T_n}{\kappa}\int_{I\setminus[0,\varepsilon_m]}\big(f(y)-\hat f_m(y)\big)^2\nu_0(dy)}.$$
\end{proof}

\begin{lemma}\label{lemma:ch4kernel}
 Let $\hat\nu_m$ and $\bar \nu_m$ the Lévy measures defined as in \eqref{eq:ch4hatf} and \eqref{eq:ch4barf}, respectively. For every $f\in \F$, there exists a Markov kernel $K$ such that 
 $$KP_{T_n}^{(\gamma^{\bar\nu_m-\nu_0},0,\bar\nu_m)}=P_{T_n}^{(\gamma^{\hat \nu_m-\nu_0},0,\hat \nu_m)}.$$
\end{lemma}
\begin{proof}
By construction, $\bar\nu_m$ and $\hat\nu_m$ coincide on $[0,\varepsilon_m]$. Let us denote by $\bar \nu_m^{\textnormal{res}}$ and $\hat\nu_m^{\textnormal{res}}$ the restriction on $I\setminus[0,\varepsilon_m]$ of $\bar\nu_m$ and $\hat\nu_m$ respectively, then it is enough to prove: $KP_{T_n}^{(\gamma^{\bar\nu_m^{\textnormal{res}}-\nu_0},0,\bar\nu_m^{\textnormal{res}})}=P_{T_n}^{(\gamma^{\hat \nu_m^{\textnormal{res}}-\nu_0},0,\hat \nu_m^{\textnormal{res}})}.$
 First of all, let us observe that the kernel $M$:
 $$M(x,A)=\sum_{j=2}^m\I_{J_j}(x)\int_A V_j(y)\nu_0(dy),\quad x\in I\setminus[0,\varepsilon_m],\quad A\in\B(I\setminus[0,\varepsilon_m])$$
 is defined in such a way that $M \bar\nu_m^{\textnormal{res}} = \hat \nu_m^{\textnormal{res}}$. 
 Indeed, for all $A\in\B(I\setminus[0,\varepsilon_m])$, 
 \begin{align}
  M\bar\nu_m^{\textnormal{res}}(A)&=\sum_{j=2}^m\int_{J_j}M(x,A)\bar\nu_m^{\textnormal{res}}(dx)=\sum_{j=2}^m \int_{J_j}\bigg(\int_A V_j(y)\nu_0(dy)\bigg)\bar\nu_m^{\textnormal{res}}(dx)\nonumber\\
           &=\sum_{j=2}^m \bigg(\int_A V_j(y)\nu_0(dy)\bigg)\nu(J_j)=\int_A \hat f_m(y)\nu_0(dy)=\hat \nu_m^{\textnormal{res}}(A). \label{eqn:M}
 \end{align}

  Observe that $(\gamma^{\bar\nu_m^{\textnormal{res}}-\nu_0},0,\bar\nu_m^{\textnormal{res}})$ and $(\gamma^{\hat \nu_m^{\textnormal{res}}-\nu_0},0,\hat \nu_m^{\textnormal{res}})$ are Lévy triplets associated with compound Poisson processes since $\bar\nu_m^{\textnormal{res}}$ and $\hat \nu_m^{\textnormal{res}}$ are finite Lévy measures.
  The Markov kernel $K$ interchanging the laws of the Lévy processes is constructed explicitly in the case of compound Poisson processes. Indeed if $\bar X$ is the compound Poisson process having Lévy measure $\bar\nu_m^{\textnormal{res}}$, then $\bar X_{t} = \sum_{i=1}^{N_t} \bar Y_{i}$, where $N_t$ is a Poisson process of intensity $\iota_m:=\bar\nu_m^{\textnormal{res}}(I\setminus [0,\varepsilon_m])$ and the $\bar Y_{i}$ are i.i.d. random variables with probability law $\frac{1}{\iota_m}\bar\nu_m^{\textnormal{res}}$.
  Moreover, given a trajectory of $\bar X$, both the trajectory $(n_t)_{t\in[0,T_n]}$ of the Poisson process $(N_t)_{t\in[0,T_n]}$ and the realizations $\bar y_i$ of $\bar Y_i$, $i=1,\dots,n_{T_n}$ are uniquely determined. This allows us to construct $n_{T_n}$ i.i.d. random variables $\hat Y_i$ as follows: For every realization $\bar y_i$ of $\bar Y_i$, we define the realization $\hat y_i$ of $\hat Y_i$ by throwing it according to the probability law $M(\bar y_i,\cdot)$. Hence, thanks to \eqref{eqn:M}, $(\hat Y_i)_i$ are i.i.d. random variables with probability law $\frac{1}{\iota_m} \hat \nu_m^{\text{res}}$. The desired Markov kernel 
$K$ 
(defined on the Skorokhod space) is then given by:
  $$
  K : (\bar X_{t})_{t\in[0,T_n]} \longmapsto \bigg(\hat X_{t} := \sum_{i=1}^{N_t} \hat Y_{i}\bigg)_{t\in[0,T_n]}.
  $$
  Finally, observe that, since 
  \begin{align*}
  \iota_m=\int_{I\setminus[0,\varepsilon_m]}\bar f_m(y)\nu_0(dy)&=\int_{I\setminus[0,\varepsilon_m]} f(y)\nu_0(dy)=\int_{I\setminus[0,\varepsilon_m]}\hat f_m(y)\nu_0(dy), 
  \end{align*}
$(\hat X_t)_{t\in[0,T_n]}$ is a compound Poisson process with Lévy measure $\hat\nu_m^{\textnormal{res}}.$  
  \end{proof}

Let us now state two lemmas needed to understand Step 4.
\begin{lemma}\label{lemma:ch4wn}
Denote by $\Wh_m^\#$ the statistical model associated with the continuous observation of a trajectory from the Gaussian white noise:
$$dy_t=\sqrt{f(t)}dt+\frac{1}{2\sqrt{T_n}\sqrt{g(t)}}dW_t,\quad t\in I\setminus [0,\varepsilon_m].$$
Then, according with the notation introduced in Section \ref{subsec:ch4parameter} and at the beginning of Section \ref{sec:ch4proofs}, we have
$$\Delta(\mathscr{N}_m,\Wh_m^\#)\leq 2\sqrt{T_n}\sup_{f\in \F} \big(A_m(f)+B_m(f)\big).$$
\end{lemma}
\begin{proof}
As a preliminary remark observe that $\Wh_m^\#$ is equivalent to the model that observes a trajectory from:
$$d\bar y_t=\sqrt{f(t)}g(t)dt+\frac{\sqrt{g(t)}}{2\sqrt{T_n}}dW_t,\quad t\in I\setminus [0,\varepsilon_m].$$
Let us denote by $\bar Y_j$ the increments of the process $(\bar y_t)$ over the intervals $J_j$,  $j=2,\dots,m$, i.e.
$$\bar Y_j:=\bar y_{v_j}-\bar y_{v_{j-1}}\sim\Nn\bigg(\int_{J_j}\sqrt{f(y)}\nu_0(dy),\frac{\nu_0(J_j)}{4T_n}\bigg)$$
and denote by $\bar{\mathscr{N}}_m$ the statistical model associated with the distributions of these increments. 
As an intermediate result, we will prove that 
\begin{equation}\label{eq:ch4normali}
\Delta(\mathscr{N}_m,\bar{\mathscr{N}}_m)\leq 2\sqrt{T_n} \sup_{f\in \F} B_m(f), \ \textnormal{ for all m}. 
\end{equation}
To that aim, remark that the experiment $\bar{\mathscr{N}}_m$ is equivalent to observing $m-1$ independent Gaussian random variables of means $\frac{2\sqrt{T_n}}{\sqrt{\nu_0(J_j)}}\int_{J_j}\sqrt{f(y)}\nu_0(dy)$, $j=2,\dots,m$ and variances identically $1$, name this last experiment $\mathscr{N}^{\#}_m$.
Hence, using also Property \ref{ch4delta0}, Facts \ref{ch4h} and \ref{fact:ch4gaussiane} we get:
\begin{align*}
\Delta(\mathscr{N}_m, \bar{\mathscr{N}}_m)\leq\Delta(\mathscr{N}_m, \mathscr{N}^{\#}_m)&\leq \sqrt{\sum_{j=2}^m\bigg(\frac{2\sqrt{T_n}}{\sqrt{\nu_0(J_j)}}\int_{J_j}\sqrt{f(y)}\nu_0(dy)-2\sqrt{T_n\nu(J_j)}\bigg)^2}.
\end{align*}

Since it is clear that $\delta(\Wh_m^\#,\bar{\mathscr{N}}_m)=0$, in order to bound $\Delta(\mathscr{N}_m,\Wh_m^\#)$ it is enough to bound $\delta(\bar{\mathscr{N}}_m,\Wh_m^\#)$. Using similar ideas as in \cite{cmultinomial} Section 8.2, we define a new stochastic process as:
$$Y_t^*=\sum_{j=2}^m\bar Y_j\int_{\varepsilon_m}^t V_j(y)\nu_0(dy)+\frac{1}{2\sqrt{T_n}}\sum_{j=2}^m\sqrt{\nu_0(J_j)}B_j(t),\quad t\in I\setminus [0,\varepsilon_m],$$
where the $(B_j(t))$ are independent centered Gaussian processes independent of $(W_t)$ and with variances
$$\textnormal{Var}(B_j(t))=\int_{\varepsilon_m}^tV_j(y)\nu_0(dy)-\bigg(\int_{\varepsilon_m}^tV_j(y)\nu_0(dy)\bigg)^2.$$
These processes can be constructed from a standard Brownian bridge $\{B(s), s\in[0,1]\}$, independent of $(W_t)$, via
$$B_i(t)=B\bigg(\int_{\varepsilon_m}^t V_i(y)\nu_0(dy)\bigg).$$
By construction, $(Y_t^*)$ is a Gaussian process with mean and variance given by, respectively:
\begin{align*}
 \E[Y_t^*]&=\sum_{j=2}^m\E[\bar Y_j]\int_{\varepsilon_m}^t V_j(y)\nu_0(dy)=\sum_{j=2}^m\bigg(\int_{J_j}\sqrt{f(y)}\nu_0(dy)\bigg)\int_{\varepsilon_m}^t V_j(y)\nu_0(dy),\\
 \textnormal{Var}[Y_t^*]&=\sum_{j=2}^m\textnormal{Var}[\bar Y_j]\bigg(\int_{\varepsilon_m}^t V_j(y)\nu_0(dy)\bigg)^2+\frac{1}{4T_n}\sum_{j=2}^m \nu_0(J_j)\textnormal{Var}(B_j(t))\\
   &= \frac{1}{4T_n}\int_{\varepsilon_m}^t \sum_{j=2}^m \nu_0(J_j) V_j(y)\nu_0(dy)= \frac{1}{4T_n}\int_{\varepsilon_m}^t  \nu_0(dy)=\frac{\nu_0([\varepsilon_m,t])}{4T_n}.
\end{align*}
One can compute in the same way the covariance of $(Y_t^*)$ finding that
$$\textnormal{Cov}(Y_s^*,Y_t^*)=\frac{\nu_0([\varepsilon_m,s])}{4 T_n}, \ \forall s\leq t.$$
We can then deduce that
$$Y^*_t=\int_{\varepsilon_m}^t \widehat{\sqrt {f}}_m(y)\nu_0(dy)+\int_{\varepsilon_m}^t\frac{\sqrt{g(s)}}{2\sqrt{T_n}}dW^*_s,\quad t\in I\setminus [0,\varepsilon_m],$$
where
$(W_t^*)$ is a standard Brownian motion and 
$$\widehat{\sqrt {f}}_m(x):=\sum_{j=2}^m\bigg(\int_{J_j}\sqrt{f(y)}\nu_0(dy)\bigg)V_j(x).$$

Applying Fact \ref{fact:ch4processigaussiani}, we get that the total variation distance between the process $(Y_t^*)_{t\in I\setminus [0,\varepsilon_m]}$ constructed from the random variables $\bar Y_j$, $j=2,\dots,m$ and the Gaussian process $(\bar y_t)_{t\in I\setminus [0,\varepsilon_m]}$ is bounded by
$$\sqrt{4 T_n\int_{I\setminus [0,\varepsilon_m]}\big(\widehat{\sqrt {f}}_m-\sqrt{f(y)}\big)^2\nu_0(dy)},$$
which gives the term in $A_m(f)$.
\end{proof}
\begin{lemma}\label{lemma:ch4limitewn}
 In accordance with the notation of Lemma \ref{lemma:ch4wn}, we have:
 \begin{equation}\label{eq:ch4wn}
\Delta(\Wh_m^\#,\Wh_n^{\nu_0})=O\bigg(\sup_{f\in\F}\sqrt{T_n\int_0^{\varepsilon_m}\big(\sqrt{f(t)}-1\big)^2\nu_0(dt)}\bigg).  
 \end{equation}
\end{lemma}
\begin{proof}
 Clearly $\delta(\Wh_n^{\nu_0},\Wh_m^\#)=0$. To show that $\delta(\Wh_m^\#,\Wh_n^{\nu_0})\to 0$, let us consider a Markov kernel $K^\#$ from $C(I\setminus [0,\varepsilon_m])$ to $C(I)$ defined as follows: Introduce a Gaussian process, $(B_t^m)_{t\in[0,\varepsilon_m]}$ with mean equal to $t$ and covariance
 $$\textnormal{Cov}(B_s^m,B_t^m)=\int_0^{\varepsilon_m}\frac{1}{4 T_n g(s)}\I_{[0,s]\cap [0,t]}(z)dz.$$
 In particular,
 $$\textnormal{Var}(B_t^m)=\int_0^t\frac{1}{4 T_n g(s)}ds.$$
 Consider it as a process on the whole of $I$ by defining $B_t^m=B_{\varepsilon_m}^m$ $\forall t>\varepsilon_m$. Let $\omega_t$ be a trajectory in $C(I\setminus [0,\varepsilon_m])$, which again we constantly extend to a trajectory on the whole of $I$. Then, we define $K^\#$ by sending the trajectory $\omega_t$ to the trajectory $\omega_t + B_t^m$. If we define $\mathbb{\tilde W}_n$ as the law induced on $C(I)$ by
 $$
 d\tilde{y}_t = h(t) dt + \frac{dW_t}{2\sqrt{T_n g(t)}}, \quad t \in I,\quad h(t) = \begin{cases}
                                                                                    1 & t \in [0, \varepsilon_m]\\
                                                                                    \sqrt{f(t)} & t \in I\setminus [0,\varepsilon_m],
                                                                                   \end{cases}$$
 then $K^\# \mathbb{W}_n^f|_{I\setminus [0,\varepsilon_m]} = \mathbb{\tilde W}_n$, where $\mathbb{W}_n^f$ is defined as in \eqref{eqn:ch4Wf}.
 By means of Fact \ref{fact:ch4processigaussiani} we deduce \eqref{eq:ch4wn}.
\end{proof}

\begin{proof}[Proof of Theorem \ref{ch4teo1}]
The proof of the theorem follows by combining the previous lemmas together:
\begin{itemize}
 \item Step 1: Let us denote by $\hat\mo_{n,m}^{\nu_0}$ the statistical model associated with the family of probabilities $(P_{T_n}^{(\gamma^{\hat\nu_m-\nu_0},0,\hat\nu_m)}:\frac{d\nu}{d\nu_0}\in\F)$. Thanks to Property \ref{ch4delta0}, Fact \ref{ch4h} and Theorem \ref{teo:ch4bound} we have that 
\begin{equation*}
 \Delta(\mo_n^{\nu_0},\hat\mo_{n,m}^{\nu_0})\leq \sqrt{\frac{T_n}{2}}\sup_{f\in \F}H(f,\hat f_m).
\end{equation*}

\item Step 2: On the one hand, thanks to Lemma \ref{lemma:ch4poisson}, one has that the statistical model associated with the family of probability $(P_{T_n}^{(\gamma^{\bar \nu_m-\nu_0},0,\bar\nu_m)}:\frac{d\nu}{d\nu_0}\in\F)$ is equivalent to $\mathscr{L}_m$. By means of Lemma \ref{lemma:ch4poissonhatf} we can bound $\Delta(\mathscr{L}_m,\hat{\mathscr{L}}_m)$. 
On the other hand it is easy to see that $\delta(\hat\mo_{n,m}^{\nu_0}, \hat{\mathscr{L}}_m)=0$. Indeed, it is enough to consider the statistics $$S: x \mapsto \bigg(\sum_{r\leq T_n}\I_{J_2}(\Delta x_r),\dots,\sum_{r\leq T_n}\I_{J_m}(\Delta x_r)\bigg)$$ since the law of the random variable $\sum_{r\leq T_n}\I_{J_j}(\Delta x_r)$ under $P_{T_n}^{(\gamma^{\hat\nu_m-\nu_0},0,\hat\nu_m)}$ is Poisson of parameter $T_n\int_{J_j}\hat f_m(y)\nu_0(dy)$ for all $j=2,\dots,m$. Finally, Lemmas \ref{lemma:ch4poisson} and \ref{lemma:ch4kernel} allows us to conclude that $\delta(\mathscr{L}_m,\hat\mo_{n,m}^{\nu_0})=0$. Collecting all the pieces together, we get
$$\Delta(\hat\mo_{n,m}^{\nu_0},\mathscr{L}_m)\leq \sup_{f\in \F}\sqrt{\frac{T_n}{\kappa}\int_{I\setminus[0,\varepsilon_m]}\big(f(y)-\hat f_m(y)\big)^2\nu_0(dy)}.$$

\item Step 3: Applying Theorem \ref{ch4teomisto} and Fact \ref{ch4hp} we can pass from the Poisson approximation given by $\mathscr{L}_m$ to a Gaussian one obtaining $$\Delta(\mathscr{L}_m,\mathscr{N}_m)=C\sup_{f\in \F}\sqrt{\sum_{j=2}^m\frac{2}{T_n\nu(J_j)}}\leq C\sqrt{\sum_{j=2}^m\frac{2\kappa}{T_n\nu_0(J_j)}}=C\sqrt{\frac{(m-1)2\kappa}{T_n\mu_m}}.$$ 

\item Step 4: Finally, Lemmas \ref{lemma:ch4wn} and \ref{lemma:ch4limitewn} allow us to conclude that:
\begin{align*}
 \Delta(\mo_n^{\nu_0},\Wh_n^{\nu_0})&=O\bigg(\sqrt{T_n}\sup_{f\in \F}\big(A_m(f)+B_m(f)+C_m\big)\bigg)\\
                            & \quad + O\bigg(\sqrt{T_n}\sup_{f\in \F}\sqrt{\int_{I\setminus{[0,\varepsilon_m]}}\big(f(y)-\hat f_m(y)\big)^2\nu_0(dy)}+\sqrt{\frac{m}{T_n\mu_m}}\bigg).
\end{align*}
\end{itemize}
\end{proof}

\subsection{Proof of Theorem \ref{ch4teo2}}

Again, before stating some technical lemmas, let us highlight the main ideas of the proof. We recall that the goal is to prove that estimating $f=\frac{d\nu}{d\nu_0}$ from the discrete observations $(X_{t_i})_{i=0}^n$ of a Lévy process without Gaussian component and having Lévy measure $\nu$ is asymptotically equivalent to estimating $f$ from the Gaussian white noise model

   $$dy_t=\sqrt{f(t)}dt+\frac{1}{2\sqrt{T_n g(t)}}dW_t,\quad g=\frac{d\nu_0}{d\Leb},\quad t\in I.$$
Reading $\mo_1 \overset{\Delta} \Longleftrightarrow \mo_2$ as $\mo_1$ is asymptotically equivalent to $\mo_2$, we have:  
   \begin{itemize}
    \item Step 1. Clearly $(X_{t_i})_{i=0}^n \overset{\Delta} \Longleftrightarrow (X_{t_i}-X_{t_{i-1}})_{i=1}^n$. Moreover, $(X_{t_i}-X_{t_{i-1}})_i\overset{\Delta} \Longleftrightarrow (\epsilon_iY_i)$ where $(\epsilon_i)$ are i.i.d Bernoulli r.v. with parameter $\alpha=\iota_m \Delta_n e^{-\iota_m\Delta_n}$, $\iota_m:=\int_{I\setminus [0,\varepsilon_m]} f(y)\nu_0(dy)$ and $(Y_i)_i$ are i.i.d. r.v. independent of $(\epsilon_i)_{i=1}^n$ and of density $\frac{ f}{\iota_m}$ with respect to ${\nu_0}_{|_{I\setminus [0,\varepsilon_m]}}$;
   
   \item 
    Step 2. $(\epsilon_iY_i)_i \overset{\Delta} \Longleftrightarrow \mathcal M(n;(\gamma_j)_{j=1}^m)$, where $\mathcal M(n;(\gamma_j)_{j=1}^m)$ is a multinomial distribution with $\gamma_1=1-\alpha$ and $\gamma_i:=\alpha\nu(J_i)$ $i=2,\dots,m$;
   
  \item     Step 3. Gaussian approximation: $\mathcal M(n;(\gamma_1,\dots\gamma_m)) \overset{\Delta} \Longleftrightarrow \bigotimes_{j=2}^m \Nn(2\sqrt{T_n\nu(J_j)},1)$;
   
    \item Step 4. $\bigotimes_{j=2}^m \Nn(2\sqrt{T_n\nu(J_j)},1)\overset{\Delta} \Longleftrightarrow (y_t)_{t\in I}$.   
 
   \end{itemize}

\begin{lemma}\label{lemma:ch4discreto}
Let $\nu_i$, $i=1,2$, be Lévy measures such that $\nu_1\ll\nu_2$ and $b_1-b_2=\int_{|y|\leq 1}y(\nu_1-\nu_2)(dy)<\infty$. Then, for all $0<t<\infty$, we have:
 $$\Big\|Q_t^{(b_1,0,\mu_1)}-Q_t^{(b_2,0,\mu_2)}\Big\|_{TV}\leq \sqrt \frac{t}{2} H(\nu_1,\nu_2).$$
\end{lemma}
\begin{proof}
For all given $t$, let $K_t$ be the Markov kernel defined as $K_t(\omega,A):=\I_A(\omega_t)$, $\forall \ A\in\B(\R)$, $\forall \ \omega\in D$. Then we have: 
 \begin{align*}
  \big\|Q_t^{(b_1,0,\nu_1)}-Q_t^{(b_2,0,\nu_2)}\big\|_{TV}&=\big\|K_tP_t^{(b_1,0,\nu_1)}-K_tP_t^{(b_2,0,\nu_2)}\big\|_{TV}\\
                                            &\leq \big\|P_t^{(b_1,0,\nu_1)}-P_t^{(b_2,0,\nu_2)}\big\|_{TV}\\
                                            &\leq \sqrt \frac{t}{2} H(\nu_1,\nu_2),
 \end{align*}
where we have used that Markov kernels reduce the total variation distance and Theorem \ref{teo:ch4bound}.
\end{proof}
 \begin{lemma}\label{lemma:ch4bernoulli}
 Let $(P_i)_{i=1}^n$, $(Y_i)_{i=1}^n$ and $(\epsilon_i)_{i=1}^n$ be samples of, respectively, Poisson random variables $\mo(\lambda_i)$, random variables with common distribution and Bernoulli random variables of parameters $\lambda_i e^{-\lambda_i}$, which are all independent. Let us denote by  $Q_{(Y_i,P_i)}$ (resp. $Q_{(Y_i,\epsilon_i)}$) the law of $\sum_{j=1}^{P_i} Y_j$ (resp., $\epsilon_i Y_i$). Then:
  \begin{equation}\label{eq:ch4lambda}
  \Big\|\bigotimes_{i=1}^n Q_{(Y_i,P_i)}-\bigotimes_{i=1}^n Q_{(Y_i,\epsilon_i)}\Big\|_{TV}\leq 2\sqrt{\sum_{i=1}^n\lambda_i^2}.
 \end{equation}
\end{lemma}
The proof of this Lemma can be found in \cite{esterESAIM}, Section 2.1.
\begin{lemma}\label{lemma:ch4troncatura}
 Let $f_m^{\textnormal{tr}}$ be the truncated function defined as follows:
 $$f_m^{\textnormal{tr}}(x)=\begin{cases}
             1 &\mbox{ if } x\in[0,\varepsilon_m]\\
             f(x)  &\mbox{ otherwise}
            \end{cases}
$$
and let $\nu_m^{\textnormal{tr}}$ (resp. $\nu_m^{\textnormal{res}}$) be the Lévy measure having $f_m^{\textnormal{tr}}$ (resp. ${f|_{I\setminus [0,\varepsilon_m]}}$) as a density with respect to $\nu_0$.
Denote by $\Qi_{n}^{\textnormal{tr},\nu_0}$ the statistical model associated with the family of probabilities  
 $\Big(\bigotimes_{i=1}^nQ_{t_i-t_{i-1}}^{(\gamma^{\nu_m^{\textnormal{tr}}-\nu_0},0,\nu_m^{\textnormal{tr}})}:\frac{d\nu_m^{\textnormal{tr}}}{d\nu_0}\in\F\Big)$ and by $\Qi_{n}^{\textnormal{res},\nu_0}$ the model associated with the family of probabilities $\Big(\bigotimes_{i=1}^nQ_{t_i-t_{i-1}}^{(\gamma^{\nu_m^{\textnormal{res}}-\nu_0},0,\nu_m^{\textnormal{res}})}:\frac{d\nu_m^{\textnormal{res}}}{d\nu_0}\in\F\Big)$. Then:
 $$\Delta(\Qi_{n}^{\textnormal{tr},\nu_0},\Qi_{n}^{\textnormal{res},\nu_0})=0.$$
\end{lemma}
\begin{proof}
 Let us start by proving that $\delta(\Qi_{n}^{\textnormal{tr},\nu_0},\Qi_{n}^{\textnormal{res},\nu_0})=0.$ 
 For that, let us consider two independent Lévy processes, $X^{\textnormal{tr}}$ and $X^0$, of Lévy triplets given by $\big(\gamma^{\nu_m^{\textnormal{tr}}-\nu_0},0,\nu_m^{\textnormal{tr}-\nu_0}\big)$ and $\big(0,0,\nu_0|_{[0,\varepsilon_m]}\big)$, respectively. Then it is clear (using the \emph{Lévy-Khintchine formula}) that the random variable $X_t^{\textnormal{tr}}- X_t^0$ is a randomization of $X_t^{\textnormal{tr}}$ (since the law of $X_t^0$ does not depend on $\nu$) having law $Q_t^{(\gamma^{\nu_m^{\textnormal{res}}-\nu_0},0,\nu_m^{\textnormal{res}})}$, for all $t\geq 0$. Similarly, one can prove that $\delta(\Qi_{n}^{\textnormal{res},\nu_0},\Qi_{n}^{\textnormal{tr},\nu_0})=0.$ 
\end{proof}
\begin{proof}[Proof of Theorem \ref{ch4teo2}]
As a preliminary remark, observe that the model $\Qi_n^{\nu_0}$ is equivalent to the one that observes the increments of $\big((x_t),P_{T_n}^{(\gamma^{\nu-\nu_0},0,\nu)}\big)$, that is, the model $\tilde\Qi_n^{\nu_0}$ associated with the family of probabilities $\Big(\bigotimes_{i=1}^nQ_{t_i-t_{i-1}}^{(\gamma^{\nu-\nu_0},0,\nu)}:\frac{d\nu}{d\nu_0}\in\F\Big)$.  
\begin{itemize}
\item Step 1: Facts \ref{ch4h}--\ref{ch4hp} and Lemma \ref{lemma:ch4discreto} allow us to write
\begin{align*}
&\Big\|\bigotimes_{i=1}^nQ_{\Delta_n}^{(\gamma^{\nu-\nu_0},0,\nu)}-\bigotimes_{i=1}^nQ_{\Delta_n}^{(\gamma^{\nu_m^{\textnormal{tr}}-\nu_0},0, \nu_m^{\textnormal{tr}})}\Big\|_{TV}\leq \sqrt{n\sqrt\frac{\Delta_n}{2}H(\nu,\nu_m^{\textnormal{tr}})}\\&=\sqrt{n\sqrt\frac{\Delta_n}{2}\sqrt{\int_0^{\varepsilon_m}\big(\sqrt{f(y)}-1\big)^2\nu_0(dy)}}.
\end{align*}
Using this bound together with Lemma \ref{lemma:ch4troncatura} and the notation therein, we get $\Delta(\Qi_n^{\nu_0}, \Qi_{n}^{\textnormal{res},\nu_0})\leq \sqrt{n\sqrt\frac{\Delta_n}{2}\sup_{f\in \F}H(f, f_m^{\textnormal{tr}})}$.
Observe that $\nu_m^{\textnormal{res}}$ is a finite Lévy measure, hence $\Big((x_t),P_{T_n}^{(\gamma^{\nu_m^{\textnormal{res}}},0,\nu_m^{\textnormal{res}})}\Big)$ is a compound Poisson process with intensity equal to $\iota_m:=\int_{I\setminus [0,\varepsilon_m]}  f(y)\nu_0(dy)$ and jumps size density $\frac{ f(x)g(x)}{\iota_m}$, for all $x\in I\setminus [0,\varepsilon_m]$ (recall that we are assuming that $\nu_0$ has a density $g$ with respect to Lebesgue). In particular, this means that $Q_{\Delta_n}^{(\gamma^{\nu_m^{\textnormal{res}}},0,\nu_m^{\textnormal{res}})}$ can be seen as the law of the random variable $\sum_{j=1}^{P_i}Y_j$ where $P_i$ is a Poisson variable of mean $\iota_m \Delta_n$, independent from $(Y_i)_{i\geq 0}$, a sequence of i.i.d. random variables with density $\frac{ fg}{\iota_m}\I_{I\setminus[0,\varepsilon_m]}$ with respect to Lebesgue. Remark also that $\iota_m$ is confined between $\kappa \nu_0\big(I\setminus [0,\varepsilon_m]\big)$ and $M\nu_0\big(I\setminus [0,\varepsilon_m]
\big)$.

Let $(\epsilon_i)_{i\geq 0}$ be a sequence of i.i.d. Bernoulli variables, independent of $(Y_i)_{i\geq 0}$, with mean  $\iota_m \Delta_n e^{-\iota_m\Delta_n}$. For $i=1,\dots,n$, denote by $Q_i^{\epsilon,f}$ the law of the variable $\epsilon_iY_i$ and by $\Qi_n^{\epsilon}$ the statistical model associated with the observations of the vector $(\epsilon_1Y_1,\dots,\epsilon_nY_n)$, i.e.
$$\Qi_n^{\epsilon}=\bigg(I^n,\B(I^n),\bigg\{\bigotimes_{i=1}^n Q_i^{\epsilon,f}:f\in\F\bigg\}\bigg).$$
Furthermore, denote by $\tilde Q_i^f$ the law of $\sum_{j=1}^{P_i}Y_j$. Then an application of Lemma \ref{lemma:ch4bernoulli} yields: 
\begin{align*}
 \Big\|\bigotimes_{i=1}^n\tilde Q_i^f&-\bigotimes_{i=1}^nQ_i^{\epsilon,f}\Big\|_{TV} \leq 2\iota_m\sqrt{n\Delta_n^2}\leq 2M\nu_0\big(I\setminus [0,\varepsilon_m]\big)\sqrt{n\Delta_n^2}.
\end{align*}
Hence, we get:
\begin{equation}\label{eq:ch4bernoulli}
\Delta(\Qi_{n}^{\textnormal{res},\nu_0},\Qi_n^{\epsilon})=O\bigg(\nu_0\big(I\setminus [0,\varepsilon_m]\big)\sqrt{n\Delta_n^2}\bigg). 
\end{equation}
Here the O depends only on $M$.
\item Step 2: Let us introduce the following random variables:
$$Z_1=\sum_{j=1}^n\I_{\{0\}}(\epsilon_jY_j); \quad Z_i=\sum_{j=1}^n\I_{J_i}(\epsilon_jY_j),\ i=2,\dots,m.$$
Observe that the law of the vector $(Z_1,\dots,Z_m)$ is multinomial $\mathcal M(n;\gamma_1,\dots,\gamma_m)$ where
$$\gamma_1=1-\iota_m \Delta_n e^{-\iota_m \Delta_n},\quad \gamma_i=\Delta_n e^{-\iota_m \Delta_n}\nu(J_i),\quad i=2,\dots,m.$$
Let us denote by $\mathcal M_n$ the statistical model associated with the observation of $(Z_1,\dots,Z_m)$. Clearly $\delta(\Qi_n^{\epsilon},\mathcal M_n)=0$. Indeed, $\mathcal M_n$ is the image experiment by the random variable $S:I^n\to\{1,\dots,n\}^{m}$ defined as 
$$S(x_1,\dots,x_n)=\Big(\#\{j: x_j=0\}; \#\big\{j: x_j\in J_2\big\};\dots;\#\big\{j: x_j\in J_m\big\}\Big),$$
where $\# A$ denotes the cardinal of the set $A$.

We shall now prove that $\delta(\mathcal M_n,\Qi_n^{\epsilon}) \leq \sup_{f\in\F}\sqrt{n\Delta_n H^2(f,\hat f_m)}$. We start by defining a discrete random variable $X^*$ concentrated at the points $0$, $x_i^*$, $i=2,\dots,m$:
$$\p(X^*=y)=\begin{cases}
             \gamma_i &\mbox{ if } y=x_i^*,\quad i=1,\dots,m,\\
             0  &\mbox{ otherwise},
            \end{cases}
$$
with the convention $x_1^*=0$. It is easy to see that $\mathcal M_n$ is equivalent to the statistical model associated with $n$ independent copies of $X^*$. Let us introduce the Markov kernel
$$
K(x_i^*, A) = \begin{cases}
               \I_A(0) & \text{if } i = 1,\\
               \int_A V_i(x) \nu_0(dx) & \text{otherwise.}
              \end{cases}
$$
Denote by $P^*$ the law of the random variable $X^*$ and by $Q_i^{\epsilon,\hat f}$ the law of a random variable $\epsilon_i \hat Y_i$ where $\epsilon_i$ is Bernoulli independent of $\hat Y_i$, with mean $\iota_m\Delta_n e^{-\iota_m\Delta_n}$ and $\hat Y_i$ has a density $\frac{\hat f_m g}{\iota_m}\I_{I\setminus[0,\varepsilon_m]}$ with respect to Lebesgue.
The same computations as in Lemma \ref{lemma:ch4kernel} prove that $KP^*=Q_i^{\epsilon,\hat f}$. Hence, thanks to Remark \ref{ch4independentkernels}, we get the equivalence between $\mathcal M_n$ and the statistical model associated with the observations of $n$ independent copies of $\epsilon_i \hat Y_i$. In order to bound $\delta(\mathcal M_n,\Qi_n^{\epsilon})$ it is enough to bound the total variation distance between the probabilities $\bigotimes_{i=1}^n Q_i^{\epsilon,f}$ and $\bigotimes_{i=1}^n  Q_i^{\epsilon,\hat f}$. Alternatively, we can bound the Hellinger distance between each of the $Q_i^{\epsilon,f}$ and $Q_i^{\epsilon,\hat f}$, thanks to Facts \ref{ch4h} and \ref{ch4hp}, which is:
\begin{align*}
 \bigg\|\bigotimes_{i=1}^nQ_i^{\epsilon,f} -\bigotimes_{i=1}^nQ_i^{\epsilon,\hat f}\bigg\|_{TV} &\leq \sqrt{\sum_{i=1}^n H^2\big(Q_i^{\epsilon,f}, Q_i^{\epsilon,\hat f}\big)}\\
 &= \sqrt{\sum_{i=1}^n \frac{1-\gamma_1}{\iota} H^2(f, \hat f_m)} \leq \sqrt{n\Delta_n H^2(f, \hat f_m)}.
\end{align*}
It follows that 
$$\delta(\mathcal M_n,\Qi_n^{\epsilon})\leq \sqrt{n\Delta_n} \sup_{f \in \F}H(f,\hat f_m).$$
 \item Step 3: Let us denote by $\mathcal N_m^*$ the statistical model associated with the observation of $m$ independent Gaussian variables $\Nn(n\gamma_i,n\gamma_i)$, $i=1,\dots,m$. Very similar computations to those in \cite{cmultinomial} yield 
 $$\Delta(\mathcal M_n,\mathcal N_m^*)=O\Big(\frac{m \ln m}{\sqrt{n}}\Big).$$
 In order to prove the asymptotic equivalence between $\mathcal M_n$ and $\mathcal N_m$ defined as in \eqref{eq:ch4n} we need to introduce some auxiliary statistical models. Let us denote by $\mathcal A_m$ the experiment obtained from $\mathcal{N}_m^*$ by disregarding the first component and by $\mathcal V_m$ the statistical model associated with the multivariate normal distribution with the same means and covariances as a multinomial distribution $\mathcal M(n,\gamma_1,\dots,\gamma_m)$.
 Furthermore, let us denote by $\mathcal N_m^{\#}$ the experiment associated with the observation of $m-1$ independent Gaussian variables $\Nn(\sqrt{n\gamma_i},\frac{1}{4})$, $i=2,\dots,m$. Clearly $\Delta(\mathcal V_m,\mathcal A_m)=0$ for all $m$: In one direction one only has to consider the projection disregarding the first component; in the other direction, it is enough to remark that $\mathcal V_m$ is the image experiment of $\mathcal A_m$ by the random variable $S:(x_2,\dots,x_m)\to (n(1-\frac{\sum_{i=2}^m x_i}{n}),x_2,\dots,x_m)$.
 Moreover, using two results contained in \cite{cmultinomial}, see Sections 7.1 and 7.2, one has that 
 $$\Delta(\mathcal A_m,\mathcal N_m^*)=O\bigg(\sqrt{\frac{m}{n}}\bigg),\quad \Delta(\mathcal A_m,\mathcal N_m^{\#})=O\bigg(\frac{m}{\sqrt n}\bigg).$$
 Finally, using Facts \ref{ch4h} and \ref{fact:ch4gaussiane} we can write
 \begin{align*}
\Delta(\mathcal N_m^{\#},\mathcal N_m)&\leq \sqrt{2\sum_{i=2}^m \Big(\sqrt{T_n\nu(J_i)}-\sqrt{T_n\nu(J_i)\exp(-\iota_m\Delta_n)}\Big)^2}\\
      &\leq\sqrt{2T_n\Delta_n^2\iota_m^3}\leq \sqrt{2n\Delta_n^3M^3\big(\nu_0\big(I\setminus [0,\varepsilon_m]\big)\big)^3}.
 \end{align*}
 To sum up, $\Delta(\mathcal M_n,\mathcal N_m)=O\Big(\frac{m \ln m}{\sqrt{n}}+\sqrt{n\Delta_n^3\big(\nu_0\big(I\setminus [0,\varepsilon_m]\big)\big)^3}\Big)$, with the $O$ depending only on $\kappa$ and $M$.
 \item Step 4: An application of Lemmas \ref{lemma:ch4wn} and \ref{lemma:ch4limitewn} yield 
 $$\Delta(\mathcal N_m,\Wh_n^{\nu_0}) \leq 2\sqrt T_n \sup_{f\in\F} \big(A_m(f)+B_m(f)+C_m(f)\big).$$
 \end{itemize}
\end{proof}

\section{Proofs of the examples}
The purpose of this section is to give detailed proofs of Examples \ref{ex:ch4esempi} and Examples \ref{ex:ch4CPP}--\ref{ex3}. As in Section \ref{sec:ch4proofs} we suppose $I\subseteq \R_+$. We start by giving some bounds for the quantities $A_m(f)$, $B_m(f)$ and $L_2(f, \hat f_m)$, the $L_2$-distance between the restriction of $f$ and $\hat f_m$ on $I\setminus[0,\varepsilon_m].$

\subsection{Bounds for $A_m(f)$, $B_m(f)$, $L_2(f, \hat{f}_m)$ when $\hat f_m$ is piecewise linear.}
In this section we suppose $f$ to be in $\F_{(\gamma, K, \kappa, M)}^I$ defined as in \eqref{ch4:fholder}. We are going to assume that the $V_j$ are given by triangular/trapezoidal functions as in \eqref{eq:ch4vj}. In particular, in this case $\hat f_m$ is piecewise linear.
\begin{lemma}\label{lemma:ch4hellinger}
 Let $0<\kappa < M$ be two constants and let $f_i$, $i=1,2$ be functions defined on an interval $J$ and such that $\kappa \leq f_i\leq M$, $i=1,2$. Then, for any measure $\nu_0$, we have:
 \begin{align*}
 \frac{1}{4 M} \int_J \big(f_1(x)-f_2(x)\big)^2 \nu_0(dx)&\leq\int_J \big(\sqrt{f_1(x)} - \sqrt{f_2(x)}\big)^2\nu_0(dx)\\
 &\leq \frac{1}{4 \kappa} \int_J \big(f_1(x)-f_2(x)\big)^2\nu_0(dx).
 \end{align*}
\end{lemma}
\begin{proof}
 This simply comes from the following inequalities:
 \begin{align*}
 \frac{1}{2\sqrt M} (f_1(x)-f_2(x)) &\leq \frac{f_1(x)-f_2(x)}{\sqrt{f_1(x)}+\sqrt{f_2(x)}} = \sqrt{f_1(x)} - \sqrt{f_2(x)}\\
 &\leq \frac{1}{2 \sqrt{\kappa}} (f_1(x)-f_2(x)).
 \end{align*}
\end{proof}

Recall that $x_i^*$ is chosen so that $\int_{J_i} (x-x_i^*) \nu_0(dx) = 0$. Consider the following Taylor expansions for $x \in J_i$:
$$
f(x) = f(x_i^*) + f'(x_i^*) (x-x_i^*) + R_i(x); \quad \hat{f}_m(x) = \hat{f}_m(x_i^*) + \hat{f}_m'(x_i^*) (x-x_i^*),
$$
where $\hat{f}_m(x_i^*) = \frac{\nu(J_i)}{\nu_0(J_i)}$ and $\hat{f}_m'(x_i^*)$ is the left or right derivative in $x_i^*$ depending whether $x < x_i^*$ or $x > x_i^*$ (as $\hat f_m$ is piecewise linear, no rest is involved in its Taylor expansion).

\begin{lemma}\label{lemma:ch4bounds}
 The following estimates hold:
 \begin{align*}
  |R_i(x)| &\leq K |\xi_i - x_i^*|^\gamma |x-x_i^*|; \\
  \big|f(x_i^*) - \hat{f}_m(x_i^*)\big| &\leq \|R_i\|_{L_\infty(\nu_0)} \text{ for } i = 2, \dots, m-1; \label{eqn:bounds}\\
  \big|f(x)-\hat{f}_m(x)\big| &\leq
  \begin{cases}
   2 \|R_i\|_{L_\infty(\nu_0)} + K |x_i^*-\eta_i|^\gamma |x-x_i^*| & \text{ if } x \in J_i, \ i = 3, \dots, m-1;\\
   C |x-\tau_i| & \text { if } x \in J_i, \ i \in \{2, m\}.
  \end{cases}
 \end{align*}
 for some constant $C$ and points $\xi_i \in J_i$, $\eta_i\in J_{i-1} \cup J_i\cup J_{i+1}$, $\tau_2 \in J_2 \cup J_3$ and $\tau_m \in J_{m-1} \cup J_m$.
\end{lemma}
\begin{proof}
 By definition of $R_i$, we have
 $$
 |R_i(x)| = \Big| \big(f'(\xi_i) - f'(x_i^*)\big)(x-x_i^*) \Big| \leq K |\xi_i - x_i^*|^\gamma |x-x_i^*|,
 $$
 for some point $\xi_i \in J_i$.
 For the second inequality,
 \begin{align*}
  |f(x_i^*)-\hat{f}_m(x_i^*)| &= \frac{1}{\nu_0(J_i)} \Big| \int_{J_i} (f(x_i^*)-f(x)) \nu_0(dx)\Big|\\
                            &= \frac{1}{\nu_0(J_i)} \bigg|\int_{J_i} R_i(x) \nu_0(dx)\bigg| \leq \|R_i\|_{L_\infty(\nu_0)},
 \end{align*}
 where in the first inequality we have used the defining property of $x_i^*$. For the third inequality, let us start by proving that for all $2 < i < m-1$,  $\hat{f}_m'(x_i^*) = f'(\chi_i)$ for some $\chi_i \in J_i\cup J_{i+1}$ (here, we are considering right derivatives; for left ones, this would be $J_{i-1} \cup J_i$). To see that, take $x\in J_i\cap [x_i^*,x_{i+1}^*]$ and introduce the function $h(x):=f(x)-l(x)$ where 
 $$l(x)=\frac{x-x_i^*}{x_{i+1}^*-x_i^*}\big(\hat f_m(x_{i+1}^*)-\hat f_m(x_i^*)\big)+\hat f_m(x_i^*).$$
 Then, using the fact that $\int_{J_i}(x-x_i^*)\nu_0(dx)=0$ joint with $\int_{J_{i+1}}(x-x_{i+1}^*)\nu_0(dx)=(x_{j+1}^*-x_j^*)\mu_m$, we get
 $$\int_{J_i}h(x)\nu_0(dx)=0=\int_{J_{i+1}}h(x)\nu_0(dx).$$
 In particular, by means of the mean theorem, one can conclude that there exist two points $p_i\in J_i$ and $p_{i+1}\in J_{i+1}$ such that
 $$h(p_i)=\frac{\int_{J_i}h(x)\nu_0(dx)}{\nu_0(J_i)}=\frac{\int_{J_{i+1}}h(x)\nu_0(dx)}{\nu_0(J_{i+1})}=h(p_{i+1}).$$
As a consequence, we can deduce that there exists $\chi_i\in[p_i,p_{i+1}]\subseteq J_i\cup J_{i+1}$ such that $h'(\chi_i)=0$, hence $f'(\chi_i)=l'(\chi_i)=\hat f_m'(x_i^*)$.
When $2 < i < m-1$, the two Taylor expansions joint with the fact that $\hat{f}_m'(x_i^*) = f'(\chi_i)$ for some $\chi_i \in J_i\cup J_{i+1}$, give
 \begin{align*}
 |f(x) - \hat{f}_m (x)| &\leq |f(x_i^*) - \hat{f}_m(x_i^*)| + |R_i(x)| + K |x_i^* - \chi_i|^\gamma |x-x_i^*|\\
 & \leq 2 \|R_i\|_{L_\infty(\nu_0)} + K |x_i^* - \chi_i|^\gamma |x-x_i^*|
 \end{align*}
 whenever $x \in J_i$ and $x > x_i^*$ (the case $x < x_i^*$ is handled similarly using the left derivative of $\hat f_m$ and $\xi_i \in J_{i-1} \cup J_i$). For the remaining cases, consider for example $i = 2$. Then $\hat{f}_m(x)$ is bounded by the minimum and the maximum of $f$ on $J_2 \cup J_3$, hence $\hat{f}_m(x) = f(\tau)$ for some $\tau \in J_2 \cup J_3$. Since $f'$ is bounded by $C = 2M +K$, one has $|f(x) - \hat{f}_m(x)| \leq C|x-\tau|$.
\end{proof}

\begin{lemma}\label{lemma:ch4abc}
 With the same notations as in Lemma \ref{lemma:ch4bounds}, the estimates for $A_m^2(f)$, $B_m^2(f)$ and $L_2(f, \hat{f}_m)^2$ are as follows:
 \begin{align*}
  L_2(f, \hat{f}_m)^2&\leq \frac{1}{4\kappa} \bigg( \sum_{i=3}^m \int_{J_i} \Big(2 \|R_i\|_{L_\infty(\nu_0)} + K |x_i^*-\eta_i|^\gamma|x-x_i^*|\Big)^2 \nu_0(dx) \\
    &\phantom{=}\ + C^2 \Big(\int_{J_2}|x-\tau_2|^2\nu_0(dx) + \int_{J_m}|x-\tau_m|^2\nu_0(dx)\Big).\\
  A_m^2(f) &= L_2\big(\sqrt{f}, \widehat{\sqrt{f}}_m\big)^2 = O\Big(L_2(f, \hat{f}_m)^2\Big)\\
  B_m^2(f) &= O\bigg( \sum_{i=2}^{m} \frac{1}{\sqrt{\kappa}} \nu_0(J_i) (2 \sqrt{M} + 1)^2 \|R_i\|_{L_\infty(\nu_0)}^2\bigg).
 \end{align*}
\end{lemma}
\begin{proof}
 The $L_2$-bound is now a straightforward application of Lemmas \ref{lemma:ch4hellinger} and \ref{lemma:ch4bounds}. The one on $A_m(f)$ follows, since if $f \in \F_{(\gamma, K, \kappa, M)}^I$ then $\sqrt{f} \in \F_{(\gamma, \frac{K}{\sqrt{\kappa}}, \sqrt{\kappa}, \sqrt{M})}^I$. In order to bound $B_m^2(f)$ write it as:
 $$B_m^2(f)=\sum_{j=1}^m \nu_0(J_j)\bigg(\frac{\int_{J_j}\sqrt{f(y)}\nu_0(dy)}{\nu_0(J_j)}-\sqrt{\frac{\nu(J_j)}{\nu_0(J_j)}}\bigg)^2=:\sum_{j=1}^m \nu_0(J_j)E_j^2.$$
 By the triangular inequality, let us bound $E_j$ by $F_j+G_j$ where:
 $$
  F_j=\bigg|\sqrt{\frac{\nu(J_j)}{\nu_0(J_j)}}-\sqrt{f(x_j^*)}\bigg| \quad \textnormal{ and }\quad  G_j=\bigg|\sqrt{f(x_j^*)}-\frac{\int_{J_j}\sqrt{f(y)}\nu_0(dy)}{\nu_0(J_j)}\bigg|.
 $$
 Using the same trick as in the proof of Lemma \ref{lemma:ch4hellinger}, we can bound:
 \begin{align*}
  F_j \leq 2 \sqrt{M} \bigg|\frac{\int_{J_j} \big(f(x)-f(x_i^*)\big)\nu_0(dx)}{\nu_0(J_j)}\bigg| \leq 2 \sqrt{M} \|R_j\|_{L_\infty(\nu_0)}.
 \end{align*}
 On the other hand,
 \begin{align*}
  G_j&=\frac{1}{\nu_0(J_j)}\bigg|\int_{J_j}\big(\sqrt{f(x_j^*)}-\sqrt{f(y)}\big)\nu_0(dy)\bigg|\\
  &=\frac{1}{\nu_0(J_j)}\bigg|\int_{J_j}\bigg(\frac{f'(x_j^*)}{2\sqrt{f(x_j^*)}}(x-x_j^*)+\tilde R_j(y)\bigg)\nu_0(dy)\bigg| \leq \|\tilde R_j\|_{L_\infty(\nu_0)},
 \end{align*}
 which has the same magnitude as $\frac{1}{\kappa}\|R_j\|_{L_\infty(\nu_0)}$.

\end{proof}

\begin{remark}
 Observe that when $\nu_0$ is finite, there is no need for a special definition of $\hat{f}_m$ near $0$, and all the estimates in Lemma \ref{lemma:ch4bounds} hold true replacing every occurrence of $i = 2$ by $i = 1$.
\end{remark}
\begin{remark}\label{rmk:nonlinear}
 The same computations as in Lemmas \ref{lemma:ch4bounds} and \ref{lemma:ch4abc} can be adapted to the general case where the $V_j$'s (and hence $\hat f_m$) are not piecewise linear. In the general case, the Taylor expansion of $\hat f_m$ in $x_i^*$ involves a rest as well, say $\hat R_i$, and one needs to bound this, as well.
\end{remark}

\subsection{Proofs of Examples \ref{ex:ch4esempi}}\label{subsec:esempi}

In the following, we collect the details of the proofs of Examples \ref{ex:ch4esempi}.

\textbf{1. The finite case:} $\nu_0\equiv \Leb([0,1])$.

Remark that in the case where $\nu_0$ if finite there are no convergence problems near zero and so we can consider the easier approximation of $f$:
\begin{equation*}
\hat f_m(x):=
 \begin{cases}
 m\theta_1 & \textnormal{if } x\in \big[0,x_1^*\big],\\
m^2\big[\theta_{j+1}(x-x_j^*)+\theta_j(x_{j+1}^*-x)\big] & \textnormal{if } x\in (x_j^*,x_{j+1}^*] \quad j = 1,\dots,m-1,\\
m\theta_m & \textnormal{if } x\in (x_m^*,1]
 \end{cases}
\end{equation*}
where $$x_j^*=\frac{2j-1}{2m},\quad J_j=\Big(\frac{j-1}{m},\frac{j}{m}\Big],\quad \theta_j=\int_{J_j}f(x)dx, \quad j=1,\dots,m.$$ 
In this case we take $\varepsilon_m = 0$ and Conditions $(C2)$ and $(C2')$ coincide:
$$\lim_{n\to\infty}n\Delta_n\sup_{f\in \F}\Big(A_m^2(f)+B_m^2(f)\Big) = 0.$$
Applying Lemma \ref{lemma:ch4abc}, we get
 $$\sup_{f\in \F} \Big(L_2(f,\hat f_m)+ A_m(f)+ B_m(f)\Big)= O\big(m^{-\frac{3}{2}}+m^{-1-\gamma}\big);$$
(actually, each of the three terms on the left hand side has the same rate of convergence).

 \textbf{2. The finite variation case: }$\frac{d\nu_0}{d\Leb}(x)=x^{-1}\I_{[0,1]}(x).$
 
 To prove that the standard choice of $V_j$ described at the beginning of Examples \ref{ex:ch4esempi} leads to $\displaystyle{\int_{\varepsilon_m}^1 V_j(x)\frac{dx}{x}=1}$, it is enough to prove that this integral is independent of $j$, since in general
 $\displaystyle{\int_{\varepsilon_m}^1 \sum_{j=2}^m V_j(x)\frac{dx}{x}=m-1}.$
 To that aim observe that, for $j=3,\dots,m-1$,
 $$\mu_m\int_{\varepsilon_m}^1 V_j(x)\nu_0(dx)=\int_{x_{j-1}^*}^{x_j^*}\frac{x-x_{j-1}^*}{x_j^*-x_{j-1}^*}\frac{dx}{x}+\int_{x_j^*}^{x_{j+1}^*}\frac{x_{j+1}^*-x}{x_{j+1}^*-x_j^*}\frac{dx}{x}.$$
 Let us show that the first addendum does not depend on $j$. We have
 $$\int_{x_{j-1}^*}^{x_j^*}\frac{dx}{x_j^*-x_{j-1}^*}=1\quad \textnormal{and}\quad -\frac{x_{j-1}^*}{x_j^*-x_{j-1}^*}\int_{x_{j-1}^*}^{x_j^*}\frac{dx}{x}=\frac{x_{j-1}^*}{x_j^*-x_{j-1}^*}\ln\Big(\frac{x_{j-1}^*}{x_j^*}\Big).$$
 Since $x_j^*=\frac{v_j-v_{j-1}}{\mu_m}$ and $v_j=\varepsilon_m^{\frac{m-j}{m-1}}$, the quantities $\frac{x_j^*}{x_{j-1}^*}$ and, hence, $\frac{x_{j-1}^*}{x_j^*-x_{j-1}^*}$ do not depend on $j$. The second addendum and the trapezoidal functions $V_2$ and $V_m$ are handled similarly.
 Thus, $\hat f_m$ can be chosen of the form

 \begin{equation*}
\hat f_m(x):=
 \begin{cases}
\quad 1 & \textnormal{if } x\in \big[0,\varepsilon_m\big],\\
\frac{\nu(J_2)}{\mu_m} & \textnormal{if } x\in \big(\varepsilon_m, x_2^*\big],\\
\frac{1}{x_{j+1}^*-x_j^*}\bigg[\frac{\nu(J_{j+1})}{\mu_m}(x-x_j^*)+\frac{\nu(J_{j})}{\mu_m}(x_{j+1}^*-x)\bigg] & \textnormal{if } x\in (x_j^*,x_{j+1}^*] \quad j = 2,\dots,m-1,\\
\frac{\nu(J_m)}{\mu_m} & \textnormal{if } x\in (x_m^*,1].
 \end{cases}
\end{equation*}
A straightforward application of Lemmas \ref{lemma:ch4bounds} and \ref{lemma:ch4abc} gives 
$$\sqrt{\int_{\varepsilon_m}^1\Big(f(x)-\hat f_m(x)\Big)^2 \nu_0(dx)} +A_m(f)+B_m(f)=O\bigg(\bigg(\frac{\ln m}{m}\bigg)^{\gamma+1} \sqrt{\ln (\varepsilon_m^{-1})}\bigg),$$
as announced.

 \textbf{3. The infinite variation, non-compactly supported case:} $\frac{d\nu_0}{d\Leb}(x)=x^{-2}\I_{\R_+}(x)$.
 Recall that we want to prove that 
 $$L_2(f,\hat f_m)^2+A_m^2(f)+B_m^2(f)=O\bigg(\frac{H(m)^{3+4\gamma}}{(\varepsilon_m m)^{2\gamma}}+\sup_{x\geq H(m)}\frac{f(x)^2}{H(m)}\bigg),$$
 for any given sequence $H(m)$ going to infinity as $m\to\infty$.
  
  Let us start by addressing the problem that the triangular/trapezoidal choice for $V_j$ is not doable. Introduce the following notation: $V_j = \tV_j + A_j$, $j = 2, \dots, m$, where the $\tV_j$'s are triangular/trapezoidal function similar to those in \eqref{eq:ch4vj}. The difference is that here, since $x_m^*$ is not defined, $\tV_{m-1}$ is a trapezoid, linear between $x_{m-2}^*$ and $x_{m-1}^*$ and constantly equal to $\frac{1}{\mu_m}$ on $[x_{m-1}^*,v_{m-1}]$ and $\tV_m$ is supported on $[v_{m-1},\infty)$, where it is constantly equal to $\frac{1}{\mu_m}$. Each $A_j$ is chosen so that:
  \begin{enumerate}
   \item It is supported on $[x_{j-1}^*, x_{j+1}^*]$ (unless $j = 2$, $j = m-1$ or $j = m$; in the first case the support is $[x_2^*, x_3^*]$, in the second one it is $[x_{m-2}^*, x_{m-1}^*]$, and $A_m \equiv 0$);
   \item ${A_j}$ coincides with $-A_{j-1}$ on $[x_{j-1}^*, x_j^*]$, $j = 3, \dots, m-1$ (so that $\sum V_j \equiv \frac{1}{\mu_n}$) and its first derivative is bounded (in absolute value) by $\frac{1}{\mu_m(x_j^* - x_{j-1}^*)}$ (so that $V_j$ is non-negative and bounded by $\frac{1}{\mu_n}$);
   \item $A_j$ vanishes, along with its first derivatives, on $x_{j-1}^*$, $x_j^*$ and $x_{j+1}^*$.
  \end{enumerate}
  We claim that these conditions are sufficient to assure that $\hat f_m$ converges to $f$ quickly enough. First of all, by Remark \ref{rmk:nonlinear}, we observe that, to have a good bound on $L_2(f, \hat f_m)$, the crucial property of $\hat f_m$ is that its first right (resp. left) derivative has to be equal to $\frac{1}{\mu_m(x_{j+1}^*-x_j^*)}$ (resp. $\frac{1}{\mu_m(x_{j}^*-x_{j-1}^*)}$) and its second derivative has to be small enough (for example, so that the rest $\hat R_j$ is as small as the rest $R_j$ of $f$ already appearing in Lemma \ref{lemma:ch4bounds}).
  
  The (say) left derivatives in $x_j^*$ of $\hat f_m$ are given by
  $$
  \hat f_m'(x_j^*) = \big(\tV_j'(x_j^*) + A_j'(x_j^*)\big) \big(\nu(J_j)-\nu(J_{j-1})\big); \quad \hat f_m''(x_j^*) = A_j''(x_j^*)\big(\nu(J_j)-\nu(J_{j-1})\big).
  $$
  Then, in order to bound $|\hat f_m''(x_j^*)|$ it is enough to bound $|A_j''(x_j^*)|$ because: 
  $$
  \big|\hat f_m''(x_j^*)\big| \leq |A_j''(x_j^*)| \Big|\int_{J_j} f(x) \frac{dx}{x^2} - \int_{J_{j-1}} f(x) \frac{dx}{x^2}\Big| \leq |A_j''(x_j^*)| \displaystyle{\sup_{x\in I}}|f'(x)|(\ell_{j}+\ell_{j-1}) \mu_m,
  $$
  where $\ell_{j}$ is the Lebesgue measure of $J_{j}$.
  
  We are thus left to show that we can choose the $A_j$'s satisfying points 1-3, with a small enough second derivative, and such that $\int_I V_j(x) \frac{dx}{x^2} = 1$. To make computations easier, we will make the following explicit choice:
  $$
  A_j(x) = b_j (x-x_j^*)^2 (x-x_{j-1}^*)^2 \quad \forall x \in [x_{j-1}^*, x_j^*),
  $$
  for some $b_j$ depending only on $j$ and $m$ (the definitions on $[x_j^*, x_{j+1}^*)$ are uniquely determined by the condition $A_j + A_{j+1} \equiv 0$ there).

  Define $j_{\max}$ as the index such that $H(m) \in J_{j_{\max}}$; it is straightforward to check that
  $$
  j_{\max} \sim m- \frac{\varepsilon_m(m-1)}{H(m)}; \quad x_{m-k}^* = \varepsilon_m(m-1) \log \Big(1+\frac{1}{k}\Big), \quad k = 1, \dots, m-2.
  $$
  One may compute the following Taylor expansions:
  \begin{align*}
  \int_{x_{m-k-1}^*}^{x_{m-k}^*} \tV_{m-k}(x) \nu_0(dx) &= \frac{1}{2} - \frac{1}{6k} + \frac{5}{24k^2} + O\Big(\frac{1}{k^3}\Big);\\
  \int_{x_{m-k}^*}^{x_{m-k+1}^*} \tV_{m-k}(x) \nu_0(dx) &= \frac{1}{2} + \frac{1}{6k} + \frac{1}{24k^2} + O\Big(\frac{1}{k^3}\Big).
  \end{align*}
  In particular, for $m \gg 0$ and $m-k \leq j_{\max}$, so that also $k \gg 0$, all the integrals $\int_{x_{j-1}^*}^{x_{j+1}^*} \tV_j(x) \nu_0(dx)$ are bigger than 1 (it is immediate to see that the same is true for $\tV_2$, as well). 
  From now on we will fix a $k \geq \frac{\varepsilon_m m}{H(m)}$ and let $j = m-k$.
  
  Summing together the conditions $\int_I V_i(x)\nu_0(dx)=1$ $\forall i>j$ and noticing that the function $\sum_{i = j}^m V_i$ is constantly equal to $\frac{1}{\mu_m}$ on $[x_j^*,\infty)$ we have:
  \begin{align*}
  \int_{x_{j-1}^*}^{x_j^*} A_j(x) \nu_0(dx) &= m-j+1 - \frac{1}{\mu_m} \nu_0([x_j^*, \infty)) - \int_{x_{j-1}^*}^{x_j^*} \tV_j(x) \nu_0(dx)\\
  &= k+1- \frac{1}{\log(1+\frac{1}{k})} - \frac{1}{2} + \frac{1}{6k} + O\Big(\frac{1}{k^2}\Big) = \frac{1}{4k} + O\Big(\frac{1}{k^2}\Big)
  \end{align*}
  Our choice of $A_j$ allows us to compute this integral explicitly:
  $$
  \int_{x_{j-1}^*}^{x_j^*} b_j (x-x_{j-1}^*)^2(x-x_j^*)^2 \frac{dx}{x^2} = b_j \big(\varepsilon_m (m-1)\big)^3 \Big(\frac{2}{3} \frac{1}{k^4} + O\Big(\frac{1}{k^5}\Big)\Big).
  $$
  In particular one gets that asymptotically
  $$
  b_j \sim \frac{1}{(\varepsilon_m(m-1))^3} \frac{3}{2} k^4 \frac{1}{4k} \sim \bigg(\frac{k}{\varepsilon_m m}\bigg)^3.
  $$
  This immediately allows us to bound the first order derivative of $A_j$ as asked in point 2: Indeed, it is bounded above by $2 b_j \ell_{j-1}^3$ where $\ell_{j-1}$ is again the length of $J_{j-1}$, namely $\ell_j = \frac{\varepsilon_m(m-1)}{k(k+1)} \sim \frac{\varepsilon_m m}{k^2}$. It follows that for $m$ big enough:
  $$
  \displaystyle{\sup_{x\in I}|A_j'(x)|} \leq \frac{1}{k^3} \ll \frac{1}{\mu_m(x_j^*-x_{j-1}^*)} \sim \bigg(\frac{k}{\varepsilon_m m}\bigg)^2.
  $$
  The second order derivative of $A_j(x)$ can be easily computed to be bounded by $4 b_j \ell_j^2$. Also remark that the conditions that $|f|$ is bounded by $M$ and that $f'$ is Hölder, say $|f'(x) - f'(y)| \leq K |x-y|^\gamma$, together give a uniform $L_\infty$ bound of $|f'|$ by $2M + K$. Summing up, we obtain:
  $$
  |\hat f_m''(x_j^*)| \lesssim b_j \ell_m^3 \mu_m \sim \frac{1}{k^3\varepsilon_m m}
  $$
  (here and in the following we use the symbol $\lesssim$ to stress that we work up to constants and to higher order terms). The leading term of the rest $\hat R_j$ of the Taylor expansion of $\hat f_m$ near $x_j^*$ is
  $$
  \hat f_m''(x_j^*) |x-x_j^*|^2 \sim |f_m''(x_j^*)| \ell_j^2 \sim \frac{\varepsilon_m m}{k^7}.
  $$

  Using Lemmas \ref{lemma:ch4bounds} and \ref{lemma:ch4abc} (taking into consideration Remark \ref{rmk:nonlinear}) we obtain
  \begin{align} 
  \int_{\varepsilon_m}^{\infty} |f(x) - \hat f_m(x)|^2 \nu_0(dx) &\lesssim \sum_{j=2}^{j_{\max}} \int_{J_j} |f(x) - \hat f_m(x)|^2 \nu_0(dx) +  \int_{H(m)}^\infty |f(x)-\hat f_m(x)|^2 \nu_0(dx) \nonumber \\ 
  &\lesssim \sum_{k=\frac{\varepsilon_m m}{H(m)}}^{m}\mu_m \bigg( \frac{(\varepsilon_m m)^{2+2\gamma}}{k^{4+4\gamma}} + \frac{(\varepsilon_m m)^2}{k^{14}}\bigg) + \frac{1}{H(m)}\sup_{x\geq H(m)}f(x)^2 \label{eq:xquadro} \\
  &\lesssim \bigg(\frac{H(m)^{3+4\gamma}}{(\varepsilon_m m)^{2+2\gamma}} + \frac{H(m)^{13}}{(\varepsilon_m m)^{10}}\bigg) + \frac{1}{H(m)}. \nonumber
  \end{align}
  It is easy to see that, since $0 < \gamma \leq 1$, as soon as the first term converges, it does so more slowly than the second one. Thus, an optimal choice for $H(m)$ is given by $\sqrt{\varepsilon_m m}$, that gives a rate of convergence:
  $$
  L_2(f,\hat f_m)^2 \lesssim \frac{1}{\sqrt{\varepsilon_m m}}.
  $$
  This directly gives a bound on $H(f, \hat f_m)$. Also, the bound on the term $A_m(f)$, which is $L_2(\sqrt f,\widehat{\sqrt{f}}_m)^2$, follows as well, since $f \in \F_{(\gamma,K,\kappa,M)}^I$ implies $\sqrt{f} \in \F_{(\gamma, \frac{K}{\sqrt\kappa}, \sqrt \kappa, \sqrt M)}^I$. Finally, the term $B_m^2(f)$ contributes with the same rates as those in \eqref{eq:xquadro}: Using Lemma \ref{lemma:ch4abc}, 
  \begin{align*}
  B_m^2(f) &\lesssim \sum_{j=2}^{\lceil m-\frac{\varepsilon_m(m-1)}{H(m)} \rceil} \nu_0(J_j) \|R_j\|_{L_\infty}^2 + \nu_0([H(m), \infty))\\
           &\lesssim \mu_m \sum_{k=\frac{\varepsilon_m (m-1)}{H(m)}}^m \Big(\frac{\varepsilon_m m}{k^2}\Big)^{2+2\gamma} + \frac{1}{H(m)}\\
           &\lesssim \frac{H(m)^{3+4\gamma}}{(\varepsilon_m m)^{2+2\gamma}} + \frac{1}{H(m)}.
  \end{align*}
 \subsection{Proof of Example \ref{ex:ch4CPP}}\label{subsec:ch4ex1} 

In this case, since $\varepsilon_m = 0$, the proofs of Theorems \ref{ch4teo1} and \ref{ch4teo2} simplify and give better estimates near zero, namely:
\begin{align}
\Delta(\mo_{n,FV}^{\Leb}, \Wh_n^{\nu_0}) &\leq C_1 \bigg(\sqrt{T_n}\sup_{f\in \F}\Big(A_m(f)+ B_m(f)+L_2(f,\hat f_m)\Big)+\sqrt{\frac{m^2}{T_n}}\bigg)\nonumber \\
\Delta(\Qi_{n,FV}^{\Leb}, \Wh_n^{\nu_0}) &\leq C_2\bigg(\sqrt{n\Delta_n^2}+\frac{m\ln m}{\sqrt{n}}+\sqrt{T_n}\sup_{f\in\F}\Big( A_m(f)+ B_m(f)+H\big(f,\hat f_m\big)\Big) \bigg) \label{eq:CPP},
\end{align}
where $C_1$, $C_2$ depend only on $\kappa,M$ and
\begin{align*} 
&A_m(f)=\sqrt{\int_0^1\Big(\widehat{\sqrt f}_m(y)-\sqrt{f(y)}\Big)^2dy},\quad
 B_m(f)=\sum_{j=1}^m\bigg(\sqrt m\int_{J_j}\sqrt{f(y)}dy-\sqrt{\theta_j}\bigg)^2.
\end{align*}

 As a consequence we get:
 \begin{align*}
  \Delta(\mo_{n,FV}^{\Leb},\Wh_n^{\nu_0})&\leq O\bigg(\sqrt{T_n}(m^{-\frac{3}{2}}+m^{-1-\gamma})+\sqrt{m^2T_n^{-1}}\bigg).
\end{align*}
 To get the bounds in the statement of Example \ref{ex:ch4CPP} the optimal choices are $m_n = T_n^{\frac{1}{2+\gamma}}$ when $\gamma \leq \frac{1}{2}$ and $m_n = T_n^{\frac{2}{5}}$ otherwise. Concerning the discrete model, we have:
 \begin{align*}
\Delta(\Qi_{n,FV}^{\Leb},\Wh_n^{\nu_0})&\leq O\bigg(\sqrt{n\Delta_n^2}+\frac{m\ln m}{\sqrt{n}}+ \sqrt{n\Delta_n}\big(m^{-\frac{3}{2}}+m^{-1-\gamma}\big)\bigg).
\end{align*}
There are four possible scenarios:
If $\gamma>\frac{1}{2}$ and $\Delta_n=n^{-\beta}$ with $\frac{1}{2}<\beta<\frac{3}{4}$ (resp. $\beta\geq \frac{3}{4}$) then the optimal choice is $m_n=n^{1-\beta}$ (resp. $m_n=n^{\frac{2-\beta}{5}}$).

If $\gamma\geq\frac{1}{2}$ and $\Delta_n=n^{-\beta}$ with $\frac{1}{2}<\beta<\frac{2+2\gamma}{3+2\gamma}$ (resp. $\beta\geq \frac{2+2\gamma}{3+2\gamma}$) then the optimal choice is $m_n=n^{\frac{2-\beta}{4+2\gamma}}$ (resp. $m_n=n^{1-\beta}$).
 \subsection{Proof of Example \ref{ch4ex2}}\label{subsec:ch4ex2}
 As in Examples \ref{ex:ch4esempi}, we let $\varepsilon_m=m^{-1-\alpha}$ and consider the standard triangular/trapezoidal $V_j$'s. In particular, $\hat f_m$ will be piecewise linear. Condition (C2') is satisfied and we have $C_m(f)=O(\varepsilon_m)$. This bound, combined with the one obtained in \eqref{eq:ch4ex2}, allows us to conclude that an upper bound for the rate of convergence of $\Delta(\Qi_{n,FV}^{\nu_0},\Wh_n^{\nu_0})$ is given by:
 $$\Delta(\Qi_{n,FV}^{\nu_0},\Wh_n^{\nu_0})\leq C \bigg(\sqrt{\sqrt{n^2\Delta_n}\varepsilon_m}+\sqrt{n\Delta_n}\Big(\frac{\ln (\varepsilon_m^{-1})}{m}\Big)^{2}+\frac{m\ln m}{\sqrt n}+\sqrt{n\Delta_n^2}\ln (\varepsilon_m^{-1}) \bigg),$$
 where $C$ is a constant only depending on the bound on $\lambda > 0$.
 
 The sequences $\varepsilon_m$ and $m$ can be chosen arbitrarily to optimize the rate of convergence. It is clear from the expression above that, if we take $\varepsilon_m = m^{-1-\alpha}$ with $\alpha > 0$, bigger values of $\alpha$ reduce the first term $\sqrt{\sqrt{n^2\Delta_n}\varepsilon_m}$, while changing the other terms only by constants. It can be seen that taking $\alpha \geq 15$ is enough to make the first term negligeable with respect to the others. In that case, and under the assumption $\Delta_n = n^{-\beta}$, the optimal choice for $m$ is $m = n^\delta$ with $\delta = \frac{5-4\beta}{14}$. In that case, the global rate of convergence is
 \begin{equation*}
 \Delta(\Qi_{n,FV}^{\nu_0},\Wh_n^{\nu_0}) = \begin{cases}
                                             O\big(n^{\frac{1}{2}-\beta} \ln n\big) & \text{if } \frac{1}{2} < \beta \leq \frac{9}{10}\\
                                             O\big(n^{-\frac{1+2\beta}{7}} \ln n\big) & \text{if } \frac{9}{10} < \beta < 1.
                                            \end{cases}
 \end{equation*}

 In the same way one can find
 $$\Delta(\mo_{n,FV}^{\nu_0},\Wh_n^{\nu_0})=O\bigg( \sqrt{n\Delta_n} \Big(\frac{\ln m}{m}\Big)^2 \sqrt{\ln(\varepsilon_m^{-1})} + \sqrt{\frac{m^2}{n\Delta_n \ln(\varepsilon_m)}} + \sqrt{n \Delta_n} \varepsilon_m \bigg).$$
 As above, we can freely choose $\varepsilon_m$ and $m$ (in a possibly different way from above). Again, as soon as $\varepsilon_m = m^{-1-\alpha}$ with $\alpha \geq 1$ the third term plays no role, so that we can choose $\varepsilon_m = m^{-2}$. Letting $\Delta_n = n^{-\beta}$, $0 < \beta < 1$, and $m = n^\delta$, an optimal choice is $\delta = \frac{1-\beta}{3}$, giving
 $$\Delta(\mo_{n,FV}^{\nu_0},\Wh_n^{\nu_0})=O\Big(n^{\frac{\beta-1}{6}} \big(\ln n\big)^{\frac{5}{2}}\Big) = O\Big(T_n^{-\frac{1}{6}} \big(\ln T_n\big)^\frac{5}{2}\Big).$$
\
 
\subsection{Proof of Example \ref{ex3}}\label{subsec:ch4ex3}
Using the computations in \eqref{eq:xquadro}, combined with $\big(f(y)-\hat f_m(y)\big)^2\leq 4 \exp(-2\lambda_0 y^3) \leq 4 \exp(-2\lambda_0 H(m)^3)$ for all $y \geq H(m)$, we obtain:
  \begin{align*}
  \int_{\varepsilon_m}^\infty \big|f(x) - \hat f_m(x)\big|^2 \nu_0(dx) &\lesssim \frac{H(m)^{7}}{(\varepsilon_m m)^{4}} + \int_{H(m)}^\infty \big|f(x) - \hat f_m(x)\big|^2 \nu_0(dx)\\
  &\lesssim \frac{H(m)^{7}}{(\varepsilon_m m)^{4}} + \frac{e^{-2\lambda_0 H(m)^3}}{H(m)}.
  \end{align*}
  As in Example \ref{ex:ch4esempi}, this bounds directly $H^2(f, \hat f_m)$ and $A_m^2(f)$. Again, the first part of the integral appearing in $B_m^2(f)$ is asymptotically smaller than the one appearing above:
  \begin{align*}
  B_m^2(f) &= \sum_{j=1}^m \bigg(\frac{1}{\sqrt{\mu_m}} \int_{J_j} \sqrt{f} \nu_0 - \sqrt{\int_{J_j} f(x) \nu_0(dx)}\bigg)^2\\
           &\lesssim \frac{H(m)^{7}}{(\varepsilon_m m)^{4}} + \sum_{k=1}^{\frac{\varepsilon_m m}{H(m)}} \bigg( \frac{1}{\sqrt{\mu_m}} \int_{J_{m-k}} \sqrt{f} \nu_0 - \sqrt{\int_{J_{m-k}} f(x) \nu_0(dx)}\bigg)^2\\
           &\lesssim \frac{H(m)^{7}}{(\varepsilon_m m)^{4}} + \frac{e^{-\lambda_0 H(m)^3}}{H(m)}.
  \end{align*}
  As above, for the last inequality we have bounded $f$ in each $J_{m-k}$, $k \leq \frac{\varepsilon_m m}{H(m)}$, with $\exp(-\lambda_0 H(m)^3)$. Thus the global rate of convergence of $L_2(f,\hat f_m)^2 + A_m^2(f) + B_m^2(f)$ is $\frac{H(m)^{7}}{(\varepsilon_m m)^{4}} + \frac{e^{-\lambda_0 H(m)^3}}{H(m)}$. 
  
  Concerning $C_m(f)$, we have $C_m^2(f) = \int_0^{\varepsilon_m} \frac{(\sqrt{f(x)} - 1)^2}{x^2} dx \lesssim \varepsilon_m^5$. To write the global rate of convergence of the Le Cam distance in the discrete setting we make the choice $H(m) = \sqrt[3]{\frac{\eta}{\lambda_0}\ln m}$, for some constant $\eta$, and obtain:
  \begin{align*}
   \Delta(\Qi_{n}^{\nu_0},\Wh_n^{\nu_0}) &= O \bigg( \frac{\sqrt{n} \Delta_n}{\varepsilon_m} + \frac{m \ln m}{\sqrt{n}} + \sqrt{n \Delta_n} \Big( \frac{(\ln m)^{\frac{7}{6}}}{(\varepsilon_m m)^2} + \frac{m^{-\frac{\eta}{2}}}{\sqrt[3]{\ln m}} \Big) + \sqrt[4]{n^2 \Delta_n \varepsilon_m^5}\bigg).
  \end{align*}
  Letting $\Delta_n = n^{-\beta}$, $\varepsilon_m = n^{-\alpha}$ and $m = n^\delta$, optimal choices give $\alpha = \frac{\beta}{3}$ and $\delta = \frac{1}{3}+\frac{\beta}{18}$. We can also take $\eta = 2$ to get a final rate of convergence:
  $$
  \Delta(\Qi_{n}^{\nu_0},\Wh_n^{\nu_0}) = \begin{cases}
                                              O\big(n^{\frac{1}{2} - \frac{2}{3}\beta}\big)& \text{if } \frac{3}{4} < \beta < \frac{12}{13}\\
                                              O\big(n^{-\frac{1}{6}+\frac{\beta}{18}} (\ln n)^{\frac{7}{6}}\big) &\text{if } \frac{12}{13} \leq \beta < 1.
                                             \end{cases}
  $$

 In the continuous setting, we have
  $$\Delta(\mo_{n}^{\nu_0},\Wh_n^{\nu_0})=O\bigg(\sqrt{n\Delta_n} \Big( \frac{(\ln m)^\frac{7}{6}}{(\varepsilon_m m)^2} + \frac{m^{-\frac{\eta}{2}}}{\sqrt[3]{\ln m}} + \varepsilon_m^{\frac{5}{2}}\Big) + \sqrt{\frac{\varepsilon_m m^2}{n\Delta_n}} \bigg).$$
  Using $T_n = n\Delta_n$, $\varepsilon_m = T_n^{-\alpha}$ and $m = T_n^\delta$, optimal choices are given by $\alpha = \frac{4}{17}$, $\delta = \frac{9}{17}$; choosing any $\eta \geq 3$ we get the rate of convergence
  $$
  \Delta(\mo_{n}^{\nu_0},\Wh_n^{\nu_0})=O\big(T_n^{-\frac{3}{34}} (\ln T_n)^{\frac{7}{6}}\big).
  $$
\appendix
\section{Background}
\subsection{Le Cam theory of statistical experiments}\label{sec:ch4lecam}
A \emph{statistical model} or \emph{experiment} is a triplet $\mo_j=(\X_j,\A_j,\{P_{j,\theta}; \theta\in\Theta\})$ where $\{P_{j,\theta}; \theta\in\Theta\}$ 
is a family of probability distributions all defined on the same $\sigma$-field $\A_j$ over the \emph{sample space} $\X_j$ and $\Theta$ is the \emph{parameter space}.
The \emph{deficiency} $\delta(\mo_1,\mo_2)$ of $\mo_1$
with respect to $\mo_2$ quantifies ``how much information we lose'' by using $\mo_1$ instead of $\mo_2$ and it is defined as
$\delta(\mo_1,\mo_2)=\inf_K\sup_{\theta\in \Theta}||KP_{1,\theta}-P_{2,\theta}||_{TV},$
 where TV stands for ``total variation'' and the infimum is taken over all ``transitions'' $K$ (see \cite{lecam}, page 18). The general definition of transition is quite involved but, for our purposes, it is enough to know that Markov kernels are special cases of transitions. By $KP_{1,\theta}$ we mean the image measure of $P_{1,\theta}$ via the Markov kernel $K$, that is
 $$KP_{1,\theta}(A)=\int_{\X_1}K(x,A)P_{1,\theta}(dx),\quad\forall A\in \A_2.$$
 The experiment $K\mo_1=(\X_2,\A_2,\{KP_{1,\theta}; \theta\in\Theta\})$ is called a \emph{randomization} of $\mo_1$ by the Markov kernel $K$. When the kernel $K$ is deterministic, that is $K(x,A)=\I_{A}S(x)$ for some random variable $S:(\X_1,\A_1)\to(\X_2,\A_2)$, the experiment $K\mo_1$ is called the \emph{image experiment by the random variable} $S$.
The Le Cam distance is defined as the symmetrization of $\delta$ and it defines a pseudometric. When $\Delta(\mo_1,\mo_2)=0$  the two statistical models are said to be \emph{equivalent}.
Two sequences of statistical models $(\mo_{1}^n)_{n\in\N}$ and $(\mo_{2}^n)_{n\in\N}$ are called \emph{asymptotically equivalent}
if $\Delta(\mo_{1}^n,\mo_{2}^n)$ tends to zero as $n$ goes to infinity. 
A very interesting feature of the Le Cam distance is that it can be also translated in terms of statistical decision theory.
Let $\D$ be any (measurable) decision space and let $L:\Theta\times \D\mapsto[0,\infty)$ denote a loss function. Let $\|L\|=\sup_{(\theta,z)\in\Theta\times\D}L(\theta,z)$. Let $\pi_i$ denote a (randomized) decision procedure in the $i$-th experiment. Denote by $R_i(\pi_i,L,\theta)$ the risk from using procedure $\pi_i$ when $L$ is the loss function and $\theta$ is the true value of the parameter. Then, an equivalent definition of the deficiency is:
\begin{align*}
 \delta(\mo_1,\mo_2)=\inf_{\pi_1}\sup_{\pi_2}\sup_{\theta\in\Theta}\sup_{L:\|L\|=1}\big|R_1(\pi_1,L,\theta)-R_2(\pi_2,L,\theta)\big|.
\end{align*}
Thus $\Delta(\mo_1,\mo_2)<\varepsilon$ means that for every procedure $\pi_i$ in problem $i$ there is a procedure $\pi_j$ in problem $j$, $\{i,j\}=\{1,2\}$, with risks differing by at most $\varepsilon$, uniformly over all bounded $L$ and $\theta\in\Theta$.
In particular, when minimax rates of convergence in a nonparametric estimation problem are obtained in one experiment, the same rates automatically hold in any asymptotically equivalent experiment. There is more: When explicit transformations from one experiment to another are obtained, statistical procedures can be carried over from one experiment to the other one.

There are various techniques to bound the Le Cam distance. We report below only the properties that are useful for our purposes. For the proofs see, e.g., \cite{lecam,strasser}.
\begin{property}\label{ch4delta0}
 Let $\mo_j=(\X,\A,\{P_{j,\theta}; \theta\in\Theta\})$, $j=1,2$, be two statistical models having the same sample space and define 
 $\Delta_0(\mo_1,\mo_2):=\sup_{\theta\in\Theta}\|P_{1,\theta}-P_{2,\theta}\|_{TV}.$
 Then, $\Delta(\mo_1,\mo_2)\leq \Delta_0(\mo_1,\mo_2)$.
\end{property}
In particular, Property \ref{ch4delta0} allows us to bound the Le Cam distance between statistical models sharing the same sample space by means of classical bounds for the total variation distance. To that aim, we collect below some useful results.
\begin{fact}\label{ch4h}
 Let $P_1$ and $P_2$ be two probability measures on $\X$, dominated by a common measure $\xi$, with densities $g_{i}=\frac{dP_{i}}{d\xi}$, $i=1,2$. Define
 \begin{align*}
  L_1(P_1,P_2)&=\int_{\X} |g_{1}(x)-g_{2}(x)|\xi(dx), \\
  H(P_1,P_2)&=\bigg(\int_{\X} \Big(\sqrt{g_{1}(x)}-\sqrt{g_{2}(x)}\Big)^2\xi(dx)\bigg)^{1/2}.
 \end{align*}
Then,
\begin{equation} 
 \|P_1-P_2\|_{TV}=\frac{1}{2}L_1(P_1,P_2)\leq H(P_1,P_2).
\end{equation}
\end{fact}
\begin{fact}\label{ch4hp}
 Let $P$ and $Q$ be two product measures defined on the same sample space: $P=\otimes_{i=1}^n P_i$, $Q=\otimes_{i=1}^n Q_i$. Then
 \begin{equation}
  H ^2(P,Q)\leq \sum_{i=1}^nH^2(P_i,Q_i).
 \end{equation}
\end{fact}
\begin{fact}\label{fact:ch4hellingerpoisson}
 Let $P_i$, $i=1,2$, be the law of a Poisson random variable with mean $\lambda_i$. Then 
 $$H^2(P_1,P_2)=1-\exp\bigg(-\frac{1}{2}\Big(\sqrt{\lambda_1}-\sqrt{\lambda_2}\Big)^2\bigg).$$
\end{fact}

\begin{fact}\label{fact:ch4gaussiane}
Let $Q_1\sim\Nn(\mu_1,\sigma_1^2)$ and $Q_2\sim\Nn(\mu_2,\sigma_2^2)$. Then
$$\|Q_1-Q_2\|_{TV}\leq \sqrt{2\bigg(1-\frac{\sigma_1^2}{\sigma_2^2}\bigg)^2+\frac{(\mu_1-\mu_2)^2}{2\sigma_2^2}}.$$
\end{fact}
\begin{fact}\label{fact:ch4processigaussiani}
For $i=1,2$, let $Q_i$, $i=1,2$, be the law on $(C,\Ci)$ of two Gaussian processes of the form
$$X^i_t=\int_{0}^t h_i(s)ds+ \int_0^t \sigma(s)dW_s,\ t\in[0,T]$$
where $h_i\in L_2(\R)$ and $\sigma\in\R_{>0}$. Then: 
 $$L_1\big(Q_1,Q_2\big)\leq \sqrt{\int_{0}^T\frac{\big(h_1(y)-h_2(y)\big)^2}{\sigma^2(s)}ds}.$$
\end{fact}

\begin{property}\label{ch4fatto3}
 Let $\mo_i=(\X_i,\A_i,\{P_{i,\theta}, \theta\in\Theta\})$, $i=1,2$, be two statistical models. 
Let $S:\X_1\to\X_2$ be a sufficient statistics
such that the distribution of $S$ under $P_{1,\theta}$ is equal to $P_{2,\theta}$. Then $\Delta(\mo_1,\mo_2)=0$. 
\end{property}
\begin{remark}\label{ch4independentkernels}
Let $P_i$ be a probability measure on $(E_i,\mathcal{E}_i)$ and $K_i$ a Markov kernel on $(G_i,\mathcal G_i)$. One can then define a Markov kernel $K$ on $(\prod_{i=1}^n E_i,\otimes_{i=1}^n \mathcal{G}_i)$ in the following way:
 $$K(x_1,\dots,x_n; A_1\times\dots\times A_n):=\prod_{i=1}^nK_i(x_i,A_i),\quad \forall x_i\in E_i,\ \forall A_i\in \mathcal{G}_i.$$
 Clearly $K\otimes_{i=1}^nP_i=\otimes_{i=1}^nK_iP_i$.
\end{remark}
Finally, we recall the following result that allows us to bound the Le Cam distance between Poisson and Gaussian variables.
\begin{theorem}\label{ch4teomisto}(See \cite{BC04}, Theorem 4)
 Let $\tilde P_{\lambda}$ be the law of a Poisson random variable $\tilde X_{\lambda}$ with mean $\lambda$. Furthermore, let $P_{\lambda}^*$ be the law of a random variable $Z^*_{\lambda}$ with Gaussian distribution $\Nn(2\sqrt{\lambda},1)$, and let $\tilde U$ be a uniform variable on $\big[-\frac{1}{2},\frac{1}{2}\big)$ independent of $\tilde X_{\lambda}$. Define
 \begin{equation}
  \tilde Z_{\lambda}=2\textnormal{sgn}\big(\tilde X_{\lambda}+\tilde U\big)\sqrt{\big|\tilde X_{\lambda}+\tilde U\big|}.
 \end{equation}
Then, denoting by $P_{\lambda}$ the law of $\tilde Z_{\lambda}$,
$$H ^2\big(P_{\lambda}, P_{\lambda}^*\big)=O(\lambda^{-1}).$$
\end{theorem}

\begin{remark}
 Thanks to Theorem \ref{ch4teomisto}, denoting by $\Lambda$ a subset of $\R_{>0}$, by $\tilde \mo$ (resp. $\mo^*$) the statistical model associated with the family of probabilities $\{\tilde P_\lambda: \lambda \in \Lambda\}$ (resp. $\{P_\lambda^* : \lambda \in \Lambda\}$), we have
 $$
 \Delta\big(\tilde \mo, \mo^*\big) \leq \sup_{\lambda \in \Lambda} \frac{C}{\lambda},
 $$
 for some constant $C$. Indeed, the correspondence associating $\tilde Z_\lambda$ to $\tilde X_\lambda$ defines a Markov kernel; conversely, associating to $\tilde Z_\lambda$ the closest integer to its square, defines a Markov kernel going in the other direction.
\end{remark}

\subsection{Lévy processes}\label{sec:ch4levy}
\begin{defn}
A stochastic process $\{X_t:t\geq 0\}$ on $\R$ defined on a probability space $(\Omega,\A,\p)$ is called a \emph{Lévy process} if the following conditions are satisfied.
\begin{enumerate}
\item $X_0=0$ $\p$-a.s.
\item For any choice of $n\geq 1$ and $0\leq t_0<t_1<\ldots<t_n$, random variables $X_{t_0}$, $X_{t_1}-X_{t_0},\dots ,X_{t_n}-X_{t_{n-1}}$are independent.
\item The distribution of $X_{s+t}-X_s$ does not depend on $s$.
\item There is $\Omega_0\in \A$ with $\p(\Omega_0)=1$ such that, for every $\omega\in \Omega_0$, $X_t(\omega)$ is right-continuous in $t\geq 0$ and has left limits in $t>0$.
\item It is stochastically continuous.
\end{enumerate}
\end{defn}
Thanks to the \emph{Lévy-Khintchine formula}, the characteristic function of any Lévy process $\{X_t\}$ can be expressed, for all $u$ in $\R$, as:
\begin{equation*}\label{caratteristica}
\E\big[e^{iuX_t}\big]=\exp\bigg(-t\Big(iub-\frac{u^2\sigma^2}{2}-\int_{\R}(1-e^{iuy}+iuy\I_{\vert y\vert \leq 1})\nu(dy)\Big)\bigg), 
\end{equation*}
where $b,\sigma\in \R$ and $\nu$ is a measure on $\R$ satisfying
$$\nu(\{0\})=0 \textnormal{ and } \int_{\R}(|y|^2\wedge 1)\nu(dy)<\infty. $$
In the sequel we shall refer to $(b,\sigma^2,\nu)$ as the characteristic triplet of the process $\{X_t\}$ and $\nu$ will be called the \emph{Lévy measure}.
This data characterizes uniquely the law of the process $\{X_t\}$. 

Let $D=D([0,\infty),\R)$ be the space of mappings $\omega$ from $[0,\infty)$ into $\R$ that are right-continuous with left limits. Define the \emph{canonical process} $x:D\to D$ by 
$$\forall \omega\in D,\quad x_t(\omega)=\omega_t,\;\;\forall t\geq 0.$$

Let $\D_t$ and $\D$ be the $\sigma$-algebras generated by $\{x_s:0\leq s\leq t\}$ and $\{x_s:0\leq s<\infty\}$, respectively (here, we use the same notations as in \cite{sato}).

By the condition (4) above, any Lévy process on $\R$ induces a probability measure $P$ on $(D,\D)$. Thus $\{X_t\}$ on the probability space $(D,\D,P)$ is identical in law with the original Lévy process. 
By saying that $(\{x_t\},P)$ is a Lévy process, we mean that $\{x_t:t\geq 0\}$ is a Lévy process under the probability measure $P$ on $(D,\D)$. For all $t>0$ we will denote $P_t$ for the restriction of $P$ to $\D_t$.
In the case where $\int_{|y|\leq 1}|y|\nu(dy)<\infty$, we set $\gamma^{\nu}:=\int_{|y|\leq 1}y\nu(dy)$.
Note that, if $\nu$ is a finite Lévy measure, then the process having characteristic triplet $(\gamma^{\nu},0,\nu)$ is a compound Poisson process.

Here and in the sequel we will denote by $\Delta x_r$ the jump of process $\{x_t\}$ at the time $r$:  
$$\Delta x_r = x_r - \lim_{s \uparrow r} x_s.$$
For the proof of Theorems \ref{ch4teo1}, \ref{ch4teo2} we also need some results on the equivalence of measures for Lévy processes. 
By the notation $\ll$ we will mean ``is absolutely continuous with respect to''.

\begin{theorem}[See \cite{sato}, Theorems 33.1--33.2 and \cite{sato2} Corollary 3.18, Remark 3.19]\label{ch4teosato}
 Let $P^1$ (resp. $P^2$) be the law induced on $(D,\D)$ by a Lévy process of characteristic triplet $(\eta,0,\nu_1)$ (resp. $(0,0,\nu_2)$), where
\begin{equation}\label{ch4gamma*}
 \eta=\int_{\vert y \vert \leq 1}y(\nu_1-\nu_2)(dy)
\end{equation}
is supposed to be finite. Then $P_t^1\ll P_t^2$ for all $t\geq 0$ if and only if $\nu_1\ll\nu_2$ and the density $\frac{d\nu_1}{d\nu_2}$ satisfies
\begin{equation}\label{ch4Sato}
 \int\bigg(\sqrt{\frac{d\nu_1}{d\nu_2}(y)}-1\bigg)^2\nu_2(dy)<\infty.
\end{equation}
Remark that the finiteness in \eqref{ch4Sato} implies that in \eqref{ch4gamma*}. When $P_t^1\ll P_t^2$, the density is
$$\frac{dP_t^1}{dP_t^2}(x)=\exp(U_t(x)),$$
with
\begin{equation}\label{ch4U}
 U_t(x)=\lim_{\varepsilon\to 0} \bigg(\sum_{r\leq t}\ln \frac{d\nu_1}{d\nu_2}(\Delta x_r)\I_{\vert\Delta x_r\vert>\varepsilon}-
\int_{\vert y\vert > \varepsilon} t\bigg(\frac{d\nu_1}{d\nu_2}(y)-1\bigg)\nu_2(dy)\bigg),\\ P^{(0,0,\nu_2)}\textnormal{-a.s.}
\end{equation}
The convergence in \eqref{ch4U} is uniform in $t$ on any bounded interval, $P^{(0,0,\nu_2)}$-a.s.
Besides, $\{U_t(x)\}$ defined by \eqref{ch4U} is a Lévy process satisfying $\E_{P^{(0,0,\nu_2)}}[e^{U_t(x)}]=1$, $\forall t\geq 0$.
\end{theorem}

Finally, let us consider the following result giving an explicit bound for the $L_1$ and the Hellinger distances between two Lévy processes of characteristic triplets of the form $(b_i,0,\nu_i)$, $i=1,2$ with $b_1-b_2=\int_{\vert y \vert \leq 1}y(\nu_1-\nu_2)(dy)$.

\begin{theorem}[See \cite{JS}]\label{teo:ch4bound}
 For any $0<T<\infty$, let $P_T^i$ be the probability measure induced on $(D,\D_T)$ by a Lévy process of characteristic triplet $(b_i,0,\nu_i)$, $i=1,2$ and suppose that $\nu_1\ll\nu_2$.
 
 If $H^2(\nu_1,\nu_2):=\int\big(\sqrt{\frac{d\nu_1}{d\nu_2}(y)}-1\big)^2\nu_2(dy)<\infty,$ then
 $$H^2(P_T^1,P_T^2)\leq \frac{T}{2}H^2(\nu_1,\nu_2).$$
\end{theorem}

We conclude the Appendix with a technical statement about the Le Cam distance for finite variation models.

\begin{lemma}\label{ch4LC}
$$\Delta(\mo_n^{\nu_0},\mo_{n,FV}^{\nu_0})=0.$$ 
\end{lemma}
\begin{proof}
Consider the Markov kernels $\pi_1$, $\pi_2$ defined as follows
\begin{equation*}
\pi_1(x,A)=\I_{A}(x^d), \quad 
 \pi_2(x,A)=\I_{A}(x-\cdot \gamma^{\nu_0}),
\quad \forall x\in D, A \in \D,
\end{equation*}
where we have denoted by $x^d$ the discontinuous part of the trajectory $x$, i.e.
$\Delta x_r = x_r - \lim_{s \uparrow r} x_s,\ x_t^d=\sum_{r \leq t}\Delta x_r$ and by $x-\cdot \gamma^{\nu_0}$ the trajectory $x_t-t\gamma{\nu_0}$, $t\in[0,T_n]$.
On the one hand we have:
\begin{align*}
 \pi_1 P^{(\gamma^{\nu-\nu_0},0,\nu)}(A)&=\int_D \pi_1(x,A)P^{(\gamma^{\nu-\nu_0},0,\nu)}(dx)=\int_D \I_A(x^d)P^{(\gamma^{\nu-\nu_0},0,\nu)}(dx)\\
 &=P^{(\gamma^{\nu},0,\nu)}(A),
\end{align*}
where in the last equality we have used the fact that, under $P^{(\gamma^{\nu-\nu_0},0,\nu)}$, $\{x_t^d\}$ is a Lévy process with characteristic triplet $(\gamma^{\nu},0,\nu)$ 
(see \cite{sato}, Theorem 19.3).
On the other hand:
\begin{align*}
 \pi_2 P^{(\gamma^{\nu},0,\nu)}(A)&=\int_D \pi_2(x,A)P^{(\gamma^{\nu_0},0,\nu)}(dx)=\int_D \I_A(x-\cdot \gamma^{\nu_0})P^{(\gamma^{\nu},0,\nu)}(dx)\\
 &=P^{(\gamma^{\nu-\nu_0},0,\nu)}(A),
\end{align*}
since, by definition, $\gamma^{\nu}-\gamma^{\nu_0}$ is equal to $\gamma^{\nu-\nu_0}$. The conclusion follows by the definition of the Le Cam distance.
\end{proof}

\subsection*{Acknowledgements}

I am very grateful to Markus Reiss for several interesting discussions and many insights; this paper would never have existed in the present form without his advice and encouragement. My deepest thanks go to the anonymous referee, whose insightful comments have greatly improved the exposition of the paper; some gaps in the proofs have been corrected thanks to his/her remarks.





\bibliographystyle{plain}
\bibliography{refs}

\end{document}